\documentclass[12pt, reqno, a4paper]{amsart}

\usepackage{amssymb}
\usepackage{hyperref}

\theoremstyle{plain}

\newtheorem*{corollary}{Corollary}
\newtheorem{lemma}{Lemma}
\newtheorem{theorem}{Theorem}
\newtheorem*{conjecture}{Conjecture}

\theoremstyle{remark}
\newtheorem*{remark}{Remark}

\theoremstyle{definition}

\newtheorem{example}{Example}

\DeclareMathOperator{\Id}{Id}

\DeclareMathOperator{\ad}{ad}
\DeclareMathOperator{\tr}{tr}
\DeclareMathOperator{\height}{ht}
\DeclareMathOperator{\End}{End}
\DeclareMathOperator{\Aut}{Aut}

\DeclareMathOperator{\length}{length}
\DeclareMathOperator{\sign}{sign}
\DeclareMathOperator{\Alt}{Alt}
\DeclareMathOperator{\Ann}{Ann}
\DeclareMathOperator{\Hom}{Hom}
\DeclareMathOperator{\PIexp}{PIexp}
\newcommand{\hatotimes}{\mathbin{\widehat{\otimes}}}

\begin{document}

\title[Graded polynomial
identities]{Graded polynomial
identities, group actions, and exponential growth of Lie algebras}

\author{A.\,S.~Gordienko}

\address{Memorial University of Newfoundland, St. John's, NL, Canada}
\email{asgordienko@mun.ca}
\keywords{Lie algebra, polynomial identity, grading, codimension, cocharacter, symmetric group, Young diagram, group action}

\begin{abstract} Consider a finite dimensional Lie algebra $L$
with an action of a finite group~$G$ over a field of characteristic~$0$.
We prove the analog of Amitsur's conjecture
on asymptotic behavior for codimensions of polynomial $G$-identities of $L$.
As a consequence, we prove the analog of Amitsur's conjecture
for graded codimensions of any finite
dimensional Lie algebra graded by a finite Abelian group.
\end{abstract}

\subjclass[2010]{Primary 17B01; Secondary 17B70, 17B40, 20C30.}
% My TeX don't say bad words on 2010 :)
%\thanks{2010 \textit{Mathematics Subject Classification.}
%Primary 17B01; Secondary 17B70, 17B40, 20C30. }
\thanks{
Supported by post doctoral fellowship
from Atlantic Association for Research
in Mathematical Sciences (AARMS), Atlantic Algebra Centre (AAC),
Memorial University of Newfoundland (MUN), and
Natural Sciences and Engineering Research Council of Canada (NSERC)}

\maketitle

\section{Introduction}

In the 1980's, a conjecture about the asymptotic behaviour
of codimensions of ordinary polynomial identities was made
by S.A.~Amitsur. Amitsur's conjecture was proved in 1999 by
A.~Giambruno and M.V.~Zaicev~\cite[Theorem~6.5.2]{ZaiGia} for associative algebras, in 2002 by M.V.~Zaicev~\cite{ZaiLie}
 for finite dimensional Lie algebras, and in 2011 by A.~Giambruno,
 I.P.~Shestakov, M.V. Zaicev for finite dimensional Jordan and alternative
 algebras~\cite{GiaSheZai}. In 2011 the author proved its analog
 for polynomial identities of finite dimensional representations of Lie
 algebras~\cite{ASGordienko}.
  Alongside with ordinary polynomial
identities of algebras, graded polynomial identities
\cite{BahtZaiGraded, BahtZaiGradedExp}  and $G$-identities are
important too~\cite{BahtZaiSehgal, BahtGiaZai}.
Therefore the question arises whether the conjecture
holds for graded and $G$-codimensions. E.~Aljadeff,  A.~Giambruno, and D.~La~Mattina
proved~\cite{AljaGiaLa, GiaLa} the analog of Amitsur's conjecture
for codimensions of graded polynomial identities of associative algebras
graded by a finite Abelian group (or, equivalently, for codimensions
of $G$-identities where $G$ is a finite Abelian group).

This article is concerned with graded codimensions
 (Theorem~\ref{TheoremMain}) and $G$-codimensions (Theorem~\ref{TheoremMainG})
 of Lie algebras.
 
 \subsection{Graded polynomial identities and their codimensions}\label{SubsectionGraded}
Let $G$ be an Abelian group.
Denote by $L(X^{\mathrm{gr}}) $ the free $G$-graded Lie
 algebra on the countable set $X^{\mathrm{gr}}=\bigcup_{g \in G}X^{(g)}$, $X^{(g)} = \{ x^{(g)}_1,
x^{(g)}_2, \ldots \}$,
over a field $F$ of characteristic~$0$,  i.e. the algebra of Lie polynomials
 in variables from $X^{\mathrm{gr}}$. The indeterminates from $X^{(g)}$ are said to be homogeneous of degree
$g$. The $G$-degree of a monomial $[x^{(g_1)}_{i_1}, \dots, x^{(g_t)}_{i_t}] \in L(X^{\mathrm{gr}})$  (all long commutators in the article are left-normed) is defined to
be $g_1 g_2 \dots g_t$, as opposed to its total degree, which is defined to be $t$. Denote by
$L(X^{\mathrm{gr}})^{(g)}$ the subspace of the algebra $L(X^{\mathrm{gr}})$ spanned by all the monomials having
$G$-degree $g$. Notice that $[L(X^{\mathrm{gr}})^{(g)}, L(X^{\mathrm{gr}})^{(h)}] \subseteq L(X^{\mathrm{gr}})^{(gh)}$, for every $g, h \in G$. It follows that
$$L(X^{\mathrm{gr}})=\bigoplus_{g\in G} L(X^{\mathrm{gr}})^{(g)}$$ is a $G$-grading.
  Let $f=f(x^{(g_1)}_{i_1}, \dots, x^{(g_t)}_{i_t}) \in L (X^{\mathrm{gr}})$.
We say that $f$ is
a \textit{graded polynomial identity} of
 a $G$-graded Lie algebra $L=\bigoplus_{g\in G}
L^{(g)}$
and write $f\equiv 0$
if $f(a^{(g_1)}_{i_1}, \dots, a^{(g_t)}_{i_t})=0$
for all $a^{(g_j)}_{i_j} \in L^{(g_j)}$, $1 \leqslant j \leqslant t$.
  The set $\Id^{\mathrm{gr}}(L)$ of graded polynomial identities of
   $L$ is
a graded ideal of $L(X^{\mathrm{gr}})$.
The case of ordinary polynomial identities is included
for the trivial group $G=\lbrace e \rbrace$.

\begin{example}\label{ExampleIdGr}
 Let $G=\mathbb Z_2 = \lbrace \bar 0, \bar 1 \rbrace$,
$\mathfrak{gl}_2(F)=\mathfrak{gl}_2(F)^{(\bar 0)}\oplus \mathfrak{gl}_2(F)^{(\bar 1)}$
where $\mathfrak{gl}_2(F)^{(\bar 0)}=\left(
\begin{array}{cc}
F & 0 \\
0 & F
\end{array}
 \right)$ and $\mathfrak{gl}_2(F)^{(\bar 1)}=\left(
\begin{array}{cc}
0 & F \\
F & 0
\end{array}
 \right)$. Then  $[x^{(\bar 0)},y^{(\bar 0)}]
\in \Id^{\mathrm{gr}}(\mathfrak{gl}_2(F))$.
\end{example}

Let $S_n$ be the $n$th symmetric group, $n\in\mathbb N$, and
 $$V^{\mathrm{gr}}_n:=\langle [x^{(g_1)}_{\sigma(1)}, x^{(g_2)}_{\sigma(2)},
 \ldots, x^{(g_n)}_{\sigma(n)}] \mid g_i \in G, \sigma \in S_n \rangle_F.$$ The non-negative integer
 $c^{\mathrm{gr}}_n(L) := \dim
  \left(\frac {V^{\mathrm{gr}}_n}{V^{\mathrm{gr}}_n \cap
   \Id^{\mathrm{gr}}(L)}\right)$
is called the $n$th {\itshape codimension of graded polynomial identities} or the $n$th {\itshape graded codimension} of~$L$.

The analog of Amitsur's conjecture for graded codimensions can be formulated
as follows.

\begin{conjecture} There exists
 $\PIexp^{\mathrm{gr}}(L):=\lim\limits_{n\to\infty} \sqrt[n]{c^\mathrm{gr}_n(L)} \in \mathbb Z_+$.
\end{conjecture}

\begin{remark}
I.B.~Volichenko~\cite{Volichenko} gave an example
of an infinite dimensional Lie algebra~$L$ with
a non-trivial polynomial identity for which the
 growth of codimensions~$c_n(L)$ of ordinary polynomial identities
 is overexponential.
 M.V.~Zaicev and S.P.~Mishchenko~\cite{ZaiMishch, VerZaiMishch}
  gave an example
of an infinite dimensional Lie algebra~$L$
 with a non-trivial polynomial identity
 such that
there exists fractional
  $\PIexp(L):=\lim\limits_{n\to\infty} \sqrt[n]{c_n(L)}$.
\end{remark}

\begin{theorem}\label{TheoremMain}
Let $L$ be a finite dimensional non-nilpotent Lie algebra
over a field $F$ of characteristic $0$, graded by a finite Abelian group $G$. Then
there exist constants $C_1, C_2 > 0$, $r_1, r_2 \in \mathbb R$, $d \in \mathbb N$
such that $C_1 n^{r_1} d^n \leqslant c^{\mathrm{gr}}_n(L) \leqslant C_2 n^{r_2} d^n$
for all $n \in \mathbb N$.
\end{theorem}
\begin{corollary}
The above analog of Amitsur's conjecture holds for such codimensions.
\end{corollary}
\begin{remark}
If $L$ is nilpotent, i.e. $[x_1, \ldots, x_p]\equiv 0$ for some $p\in\mathbb N$, then  $V^{\mathrm{gr}}_n \subseteq \Id^{\mathrm{gr}}(L)$ and $c^{\mathrm{gr}}_n(L)=0$ for all $n \geqslant p$.
\end{remark}

Theorem~\ref{TheoremMain} will be obtained as a consequence of Theorem~\ref{TheoremMainG}
in Subsection~\ref{SubsectionDuality}.

\subsection{Polynomial $G$-identities and their codimensions}\label{SubsectionG}
Analogously, one can consider polynomial $G$-identities
for any group $G$.
We use the exponential notation
  for the action of a group and its group algebra.
  We say that a Lie algebra $L$ is a \textit{Lie algebra with $G$-action}
  or a \textit{Lie $G$-algebra}
   if there is a fixed linear representation $G \to
  \mathrm{GL}(L)$ such that $[a,b]^g=[a^g,b^g]$ for all $a,b \in L$ and $g \in G$.
 Denote by $L( X | G)$
the free Lie algebra over $F$ with free formal
 generators $x^g_j$, $j\in\mathbb N$,
 $g \in G$.
  Define $(x^g_j)^h := x^{hg}_j$
 for $h \in G$.
 Let $X := \lbrace x_1, x_2, x_3, \ldots \rbrace$
 where $x_j := x_j^1$, $1\in G$.
 Then $L( X | G)$ becomes the free $G$-algebra with
 free generators $x_j$, $j \in \mathbb N$.
 Let $L$ be a Lie $G$-algebra over $F$. A polynomial
  $f(x_1, \ldots, x_n)\in L( X|G)$
 is a \textit{$G$-identity} of $L$ if $f(a_1, \ldots, a_n)=0$
for all $a_i \in L$. The set $\Id^{G}(L)$ of all $G$-identities
of $L$ is an ideal in $L( X | G)$ invariant under $G$-action.

\begin{example}\label{ExampleIdG} Consider $\psi \in \Aut(\mathfrak{gl}_2(F))$
defined by the formula $$\left(
\begin{array}{cc}
a & b \\
c & d
\end{array}
 \right)^\psi := \left(
\begin{array}{rr}
a & -b \\
-c & d
\end{array}
 \right).$$
Then $[x+x^{\psi},y+y^{\psi}]\in \Id^{G}(\mathfrak{gl}_2(F))$
where
$G=\langle \psi \rangle \cong \mathbb Z_2$.
\end{example}

Denote by $V^G_n$ the space of all multilinear $G$-polynomials
in $x_1, \ldots, x_n$, i.e.
$$V^{G}_n = \langle [x^{g_1}_{\sigma(1)},
x^{g_2}_{\sigma(2)},\ldots, x^{g_n}_{\sigma(n)}]
\mid g_i \in G, \sigma\in S_n \rangle_F.$$
Then the number $c^G_n(L):=\dim\left(\frac{V^G_n}{V^G_n \cap \Id^G(L)}\right)$
is called the $n$th \textit{codimension of polynomial $G$-identities} 
or the $n$th \textit{$G$-codimension} of $L$.

\begin{remark}
As in the case of associative algebras~\cite[Lemma~10.1.3]{ZaiGia},
 we have
$$c_n(L) \leqslant c_n^G(L) \leqslant |G|^n c_n(L).$$
 Here
$c_n(L)=c_n^{\{e\}}(L)$ are ordinary codimensions.
\end{remark}

Also we have the following upper bound:

\begin{lemma}\label{LemmaCodimDim}
Let $L$ be a finite dimensional Lie algebra with $G$-action
over any field $F$ and let $G$ be any group. Then
$c_n^G(L) \leqslant (\dim L)^{n+1}$.
\end{lemma}
\begin{proof}
Consider $G$-polynomials as $n$-linear maps from $L$ to $L$.
Then we have a natural map $V^{G}_n \to \Hom_{F}(L^{{}\otimes n}; L)$
with the kernel $V^{G}_n \cap \Id^G(L)$
that leads to the embedding $$\frac{V^{G}_n}{V^{G}_n \cap \Id^G(L)}
\hookrightarrow \Hom_{F}(L^{{}\otimes n}; L).$$
Thus $$c^G_n(L)=\dim \left(\frac{V^{G}_n}{V^{G}_n \cap \Id^G(L)}\right)
\leqslant \dim \Hom_{F}(L^{{}\otimes n}; L)=(\dim L)^{n+1}.$$
\end{proof}

The analog of Amitsur's conjecture for $G$-codimensions can be formulated
as follows.

\begin{conjecture} There exists
 $\PIexp^G(L):=\lim\limits_{n\to\infty} \sqrt[n]{c^G_n(L)} \in \mathbb Z_+$.
\end{conjecture}

 \begin{theorem}\label{TheoremMainG}
Let $L$ be a finite dimensional non-nilpotent Lie algebra
over a field $F$ of characteristic $0$. Suppose a finite
  group $G$ not necessarily
Abelian acts on $L$. Then there exist
 constants $C_1, C_2 > 0$, $r_1, r_2 \in \mathbb R$,
  $d \in \mathbb N$ such that
   $C_1 n^{r_1} d^n \leqslant c^{G}_n(L)
    \leqslant C_2 n^{r_2} d^n$ for all $n \in \mathbb N$.
\end{theorem}

\begin{corollary}
The above analog of Amitsur's conjecture holds for
 such codimensions.
\end{corollary}

\begin{remark}
If $L$ is nilpotent, i.e. $[x_1, \ldots, x_p]\equiv 0$ for some $p\in\mathbb N$, then, by the Jacobi identity,
$V^{G}_n \subseteq \Id^{G}(L)$ and $c^G_n(L)=0$ for all $n \geqslant p$.
\end{remark}

\begin{remark}
The theorem is still true if we allow $G$ to act not only by automorphisms,
but by anti-automorphisms too, i.e. if $G = G_0 \cup G_1$ such that
$[a,b]^g=[a^g, b^g]$ for all $a,b \in L$, $g \in G_0$ and
$[a,b]^g=[b^g, a^g]$ for all $a,b \in L$, $g \in G_1$. Indeed,
we can replace $G$ with $\tilde G = G_0 \cup (-G_1)$
where $[a,b]^{-g}=-[a,b]^g=-[b^g,a^g]=[a^{-g},b^{-g}]$ for all $(-g) \in (-G_1)$.
Then $\tilde G$ acts on $L$ by automorphisms only. Moreover,
$n$-linear functions from $L$ to $L$ that correspond to polynomials from $P^G_n$ and $P^{\tilde G}_n$, are the same. Thus
$$c^G_n(L)=\dim \left(\frac{V^{G}_n}{V^{G}_n \cap \Id^G(L)}\right)=
\dim \left(\frac{V^{\tilde G}_n}{ V^{\tilde G}_n \cap \Id^{\tilde G}(L)}\right)
= c^{\tilde G}_n(L)
$$ has the desired asymptotics.
\end{remark}

Theorem~\ref{TheoremMainG} is proved in Sections~\ref{SectionUpper}--\ref{SectionLower}.

\subsection{Duality between group gradings and group actions}\label{SubsectionDuality}

If $F$ is an algebraically closed field of
characteristic $0$ and $G$ is finite Abelian,
 there exists a well known duality between
 $G$-gradings and
  $\widehat G$-actions
  where $\widehat G = \Hom(G,F^{*}) \cong G$.
    Details of the application of this duality
    to polynomial identities
  can be found, e.g., in~\cite[Chapters~3 and 10]{ZaiGia}.

    A character $\psi \in \widehat G$ acts on $L$ in the natural
way: $(a_g)^\psi = \psi(g) a_g$ for all $g \in G$
and $a_g \in L^{(g)}$. Conversely, if $L$ is a $\widehat G$-algebra,
then $L^{(g)} = \lbrace a \in L \mid a^\psi = \psi(g)a \text{ for all } \psi \in
\widehat G \rbrace$ defines a $G$-grading on $L$.

Note that if $G$ is finite Abelian, then $L(X^{\mathrm{gr}})$ is a free $\widehat G$-algebra with
 free generators $y_j = \sum_{g\in G} x^{(g)}_j$.
Thus there exists an isomorphism $\varepsilon \colon L( X | \widehat G)
 \to L(X^{\mathrm{gr}})$
 defined by $\varepsilon(x_j)=\sum_{g\in G} x^{(g)}_j$,
that preserves $\widehat G$-action and $G$-grading.
The isomorphism has the property $\varepsilon((x_j)^{e_g})
=x^{(g)}_j$ where $e_g := \frac{1}{|G|}\sum_{\psi}(\psi(g))^{-1}\psi$
 is one of the minimal idempotents of $F\widehat G$ defined above.

\begin{lemma}\label{LemmaGradAction}
Let $L$ be a $G$-graded Lie algebra where $G$ is a finite Abelian group.
Consider the corresponding $\widehat G$-action on $L$. Then
 \begin{enumerate}
\item $\varepsilon(\Id^{\widehat G}(L))=\Id^{\mathrm{gr}}(L)$;
\item $c^{\widehat G}_n(L)=c^{\mathrm{gr}}_n(L)$.
\end{enumerate}
\end{lemma}
\begin{proof}
The first assertion is evident. The second assertion follows
from the first one and the equality
$\varepsilon(V^{\widehat G}_n)=V^{\mathrm{gr}}_n$.
\end{proof}
\begin{remark}
Note that $\mathbb Z_2$-grading in Example~\ref{ExampleIdGr}
corresponds to $\mathbb Z_2$-action in Example~\ref{ExampleIdG}.
\end{remark}

\begin{proof}[Proof of Theorem~\ref{TheoremMain}]
Codimensions do not change upon an extension of the base field.
The proof is analogous to the cases of ordinary codimensions of
associative~\cite[Theorem~4.1.9]{ZaiGia} and
Lie algebras~\cite[Section~2]{ZaiLie}.
Thus without loss of generality we may assume
 $F$ to be algebraically closed.
In virtue of Lemma~\ref{LemmaGradAction}, Theorem~\ref{TheoremMain}
 is an immediate consequence of Theorem~\ref{TheoremMainG}.
\end{proof}

\subsection{Formula for the PI-exponent}
Theorem~\ref{TheoremMainG} is formulated for an arbitrary
field $F$ of characteristic $0$, but without loss of generality
we may assume that $F$ is algebraically closed.

Fix a Levi decomposition $L=B\oplus R$ where $B$ is a maximal semisimple
subalgebra of $L$ and $R$ is the solvable radical of $L$.
Note that $R$ is invariant under $G$-action.
By~\cite[Theorem~1, Remark~3]{Taft}, we can choose $B$
invariant under $G$-action too.

We say that $M$ is an \textit{$L$-module with $G$-action}
if $M$ is both left $L$- and $FG$-module, and
 $(a\cdot v)^g = a^g \cdot v^g$ for all $a\in L$, $v \in M$
and $g \in G$. There is a natural $G$-action on $\End_F(M)$
defined by $\psi^g m = (\psi m^{g^{-1}})^g$, $m\in M$, $g\in G$,
$\psi \in \End_F(M)$.
Note that $L \to \mathfrak{gl}(M)$ is a homomorphism
of $FG$-modules.
 Such module $M$ is \textit{irreducible}
if for any $G$- and $L$-invariant
subspace $M_1 \subseteq M$ we have either $M_1=0$ or $M_1=M$.
Each $G$-invariant ideal in $L$ can be regarded as
a left $L$-module with $G$-action under the adjoint representation
of $L$.

Consider $G$-invariant ideals $I_1, I_2, \ldots, I_r$,
$J_1, J_2, \ldots, J_r$, $r \in \mathbb Z_+$, of the algebra $L$ such that $J_k \subseteq I_k$,
satisfying the conditions
\begin{enumerate}
\item $I_k/J_k$ is an irreducible $L$-module with $G$-action;
\item for any $G$-invariant $B$-submodules $T_k$
such that $I_k = J_k\oplus T_k$, there exist numbers
$q_i \geqslant 0$ such that $$[[T_1, \underbrace{L, \ldots, L}_{q_1}], [T_2, \underbrace{L, \ldots, L}_{q_2}], \ldots, [T_r,
 \underbrace{L, \ldots, L}_{q_r}]] \ne 0.$$
\end{enumerate}

Let $M$ be an $L$-module. Denote by $\Ann M$ its annihilator in $L$.
Let $$d(L) := \max \left(\dim \frac{L}{\Ann(I_1/J_1) \cap \dots \cap \Ann(I_r/J_r)}
\right)$$
where the maximum is found among all $r \in \mathbb Z_+$ and all $I_1, \ldots, I_r$, $J_1, \ldots, J_r$
satisfying Conditions 1--2. We claim
that $\PIexp^G(L)=d(L)$ and prove
Theorem~\ref{TheoremMainG} for $d=d(L)$.

\subsection{Examples} Now we give several examples.
\begin{example}\label{ExampleGSimple}
Let $L$ be a finite dimensional $G$-simple Lie algebra
over an algebraically closed field $F$ of characteristic $0$
where
$G$ is a finite group.
Then there exist $C > 0$ and $r \in \mathbb R$ such that
$C n^r (\dim L)^n \leqslant c_n^G(L) \leqslant (\dim L)^{n+1}$.
\end{example}
\begin{proof}
The upper bound follows from Lemma~\ref{LemmaCodimDim}.
Consider $G$-invariant $L$-modules $I_1 = L$ and $J_1=0$.
Then  $I_1/J_1$ is an irreducible $L$-module,
 $\Ann(I_1/J_1)=0$ since a $G$-simple algebra has zero center,
 and $\dim(L/\Ann(I_1/J_1))=\dim L$. Thus $d(L) \geqslant \dim L$
 and by Theorem~\ref{TheoremMainG} we obtain the lower bound.
\end{proof}

\begin{example}\label{ExampleGrSimple}
Let $L$ be a finite dimensional simple $G$-graded Lie algebra
over an algebraically closed field $F$ of characteristic $0$ where
$G$ is a finite Abelian group.
Then there exist $C > 0$ and $r \in \mathbb R$ such that
$C n^r (\dim L)^n \leqslant c_n^{\mathrm{gr}}(L) \leqslant (\dim L)^{n+1}$.
\end{example}
\begin{proof}
This follows from~Example~\ref{ExampleGSimple}
and Lemma~\ref{LemmaGradAction}.
\end{proof}

\begin{example}\label{ExampleGSolvable}
Let $L$ be a finite dimensional Lie algebra with $G$-action
over any field $F$ of characteristic $0$ such
 that $\PIexp^G(L) \leqslant 2$ where
$G$ is a finite group. Then $L$ is solvable.
\end{example}
\begin{proof}
It is sufficient to prove the statement for an algebraically
closed field $F$.
(See the remark before Theorem~\ref{TheoremMainG}.)
 Consider the $G$-invariant Levi decomposition
$L=B\oplus R$. If $B \ne 0$, there exists a $G$-simple Lie subalgebra $B_1 \subseteq L$,
$\dim B_1 \geqslant 3$ and $\PIexp^G(L) = d(L) \geqslant 3$
by Example~\ref{ExampleGSimple}.
We get a contradiction. Hence $L=R$ is a solvable algebra.
\end{proof}

Analogously, we derive Example~\ref{ExampleGrSolvable}
from Example~\ref{ExampleGrSimple}.

\begin{example}\label{ExampleGrSolvable}
Let $L$ be a finite dimensional $G$-graded Lie algebra
over any field $F$ of characteristic $0$
 such that $\PIexp^\mathrm{gr}(L)
 \leqslant 2$ where
$G$ is a finite Abelian group. Then $L$ is solvable.
\end{example}

\begin{example}\label{ExampleGSemiSimple}
Let $L=B_1 \oplus \ldots \oplus B_s$ be a finite dimensional semisimple Lie $G$-algebra
over an algebraically closed field $F$ of characteristic $0$
where
$G$ is a finite group and $B_i$ are $G$-minimal ideals.
 Let $d:=\max_{1 \leqslant i \leqslant s} \dim B_i$.
Then there exist $C_1, C_2 > 0$ and $r_1, r_2 \in \mathbb R$ such that
$C_1 n^{r_1} d^n \leqslant c_n^G(L) \leqslant C_2 n^{r_2} d^n$.
\end{example}
\begin{proof}
Note that if $I$ is a $G$-simple ideal of $L$, then $[I,L] \ne 0$ and
hence $[I,B_i] \ne 0$
for some $1 \leqslant i \leqslant s$.
However $[I,B_i] \subseteq B_i \cap I$ is a $G$-invariant ideal.
Thus $I=B_i$. And if $I$ is a $G$-invariant ideal of $L$,
then it is semisimple and each of its simple components coincides with
one of $B_i$. Thus if $I \subseteq J$ are $G$-invariant ideals of $L$
and $I/J$ is irreducible,
then $I = B_i \oplus J$ for some $1 \leqslant i \leqslant s$
and $\dim(L/\Ann(I/J))=\dim B_i$. Suppose $I_1, \ldots, I_r$, $J_1, \ldots, J_r$ satisfy Conditions 1--2.
Let $I_k = B_{i_k}\oplus J_k$,
 $1 \leqslant k \leqslant r$.
 Then
  $$[[B_{i_1}, L, \ldots, L], [B_{i_2}, L, \ldots, L], \ldots, [B_{i_r}, L, \ldots, L]]\ne 0$$
  for some number of copies of $L$.
 Hence $i_1 = \ldots = i_r$ and $$\dim \frac{L}{\Ann(I_1/J_1) \cap \dots \cap \Ann(I_r/J_r)} = \dim B_{i_1}.$$ Therefore, $d(L)=\max_{1 \leqslant i \leqslant s} \dim B_i$
and the result follows from Theorem~\ref{TheoremMainG}.
\end{proof}

\begin{example}\label{ExampleGradSemiSimple}
Let $L=B_1 \oplus \ldots \oplus B_s$ be a finite dimensional semisimple
$G$-graded Lie algebra
over an algebraically closed field $F$ of characteristic $0$ where
$G$ is a finite Abelian group and $B_i$ are minimal graded ideals.
 Let $d:=\max_{1 \leqslant i \leqslant s} \dim B_i$.
Then there exist $C_1, C_2 > 0$ and $r_1, r_2 \in \mathbb R$ such that
$C_1 n^{r_1} d^n \leqslant c_n^{\mathrm{gr}}(L) \leqslant C_2 n^{r_2} d^n$.
\end{example}
\begin{proof}
This follows from~Example~\ref{ExampleGSemiSimple}
and Lemma~\ref{LemmaGradAction}.
\end{proof}

\begin{example}\label{ExampleMetabelian}
Let $m\in \mathbb N$, $G\subseteq S_m$
and $O_i$ be the orbits of $G$-action on
$$\lbrace1, 2, \ldots, m \rbrace = \coprod_{i=1}^s O_i.$$
 Denote $$d:=\max_{1\leqslant i \leqslant s} |O_i|.$$
Let $L$ be the Lie algebra
over any field $F$ of characteristic $0$ with basis $a_1, \ldots, a_m$, $b_1,
\ldots, b_m$, $\dim L = 2m$, and multiplication
defined by formulas $[a_i, a_j]=[b_i,b_j]=0$ and
 $$[a_i, b_j] = \left\lbrace \begin{array}{rrr} b_j & \text{if} &
  i = j,\\
0 & \text{if} & i \ne j.
\end{array} \right.$$
 Suppose $G$ acts on $L$
as follows:
 $(a_i)^{\sigma} =a_{\sigma(i)}$ and
 $(b_j)^{\sigma} =b_{\sigma(j)}$ for $\sigma \in G$.
Then there exist $C_1, C_2 > 0$ and $r_1, r_2 \in \mathbb R$ such that
$$C_1 n^{r_1} d^n \leqslant c_n^{G}(L) \leqslant C_2 n^{r_2} d^n.$$
In particular, if $$G=\langle\tau\rangle \cong \mathbb Z_m = \mathbb Z/(m\mathbb Z)
=\lbrace \bar 0, \bar 1, \ldots, \overline{m-1} \rbrace$$
where $\tau = (1\,2\, 3\, \ldots\, m)$ (a cycle),
then
$$C_1 n^{r_1} m^n \leqslant c_n^{G}(L) \leqslant C_2 n^{r_2} m^n.$$
However, $c_n(L)=n-1$ for all $n\in\mathbb N$.
\end{example}
\begin{proof} If $K \supseteq F$ is a larger field, then $K \mathbin{\otimes_F} L$
is defined by the same formulas as $L$.
Since $c^G_n(L)=c^{G,K}_n(K \mathbin{\otimes_F} L)$
(see the remark before Theorem~\ref{TheoremMainG}),
we may assume $F$ to be
algebraically closed.

Let $B_i:=\langle b_j \mid j \in O_i \rangle_F$,
$1 \leqslant i \leqslant s$.
Suppose $I$ is a $G$-invariant ideal of $L$.
 If $b_i \in I$, then
 $b_{\sigma(i)}=(b_i)^\sigma \in I$ for all $\sigma \in G$.
 Thus if $i \in O_j$, then $b_k \in I$ for all $k \in O_j$.
 Let $c:=\sum\limits_{i=1}^m (\alpha_i a_i + \beta_i b_i) \in I$
 for some $\alpha_i, \beta_i \in F$.
 Then $\beta_i b_i = [a_i, c] \in I$ for all $1\leqslant i \leqslant m$ too.
Therefore, $I=A_0 \oplus B_{i_1} \oplus \ldots \oplus B_{i_k}$
for some $1 \leqslant i_j \leqslant s$ and $A_0 \subseteq \langle
a_1, \ldots, a_m \rangle_F$.

If $I,J \subseteq L$ are $G$-invariant ideals, then
$J \subseteq J+[L,L]\cap I \subseteq I$ is a $G$-invariant ideal
too. Suppose $I/J$ is irreducible. Then either $[L,L]\cap I \subseteq J$
and $\Ann(I/J)=L$ or $I \subseteq J + [L,L]$
where $[L,L]=\langle b_1, \ldots, b_m \rangle_F$.
Thus $\Ann(I/J)\ne L$ implies $J=A_0 \oplus B_{i_1} \oplus \ldots \oplus B_{i_k}$
and $I=B_\ell \oplus J$ for some $1 \leqslant \ell \leqslant s$.
In this case $\dim(L/\Ann(I/J))=|O_\ell|$.

Note that if $I_1 = B_{i_1}\oplus J_1$
and $I_2 = B_{i_2}\oplus J_2$, then $$[[B_{i_1}, L, \ldots, L], [B_{i_2}, L, \ldots, L]]=0.$$
Thus $I_1, \ldots, I_r$, $J_1, \ldots, J_r$ can satisfy Conditions 1--2
only if $r=1$. Hence $$d(L)=\max_{1 \leqslant i \leqslant s} |O_\ell|$$
and by Theorem~\ref{TheoremMainG} we obtain the bounds.

Consider the ordinary polynomial identities.
Using the Jacobi identity, any monomial in $V_n$
can be rewritten as a linear combination of
left-normed commutators $[x_1, x_j, x_{i_3},
\ldots, x_{i_n}]$. Since the polynomial identity $$[[x,y],[z,t]]\equiv 0$$
holds in $L$, we may assume that $i_3 < i_4 < \ldots < i_n$.
Note that $f_j=[x_1, x_j, x_{i_3},
\ldots, x_{i_n}]$, $2\leqslant j \leqslant n$, are linearly independent modulo $\Id(L)$.
Indeed, if $\sum_{k=2}^n \alpha_k f_k \equiv 0$, $\alpha_k \in F$,
 then we substitute $x_j=b_1$ and $x_i=a_1$ for $i\ne j$.
 Only $f_j$ does not vanish.
 Hence $\alpha_j=0$ and $c_n(L)=n-1$.
\end{proof}

\begin{example}\label{ExampleMetabelianGr}
Let $m\in \mathbb N$, $L=\bigoplus\limits_{\bar k \in \mathbb Z_m} L^{(\bar k)}$
 be the $\mathbb Z_m$-graded Lie algebra with
 $L^{(\bar k)}= \langle c_{\bar k}, d_{\bar k} \rangle_F$,
  $\dim L^{(\bar k)} = 2$, multiplication
  $[c_{\bar\imath}, c_{\bar\jmath}]=
  [d_{\bar\imath},d_{\bar\jmath}]=0$ and
  $[c_{\bar\imath}, d_{\bar\jmath}] =
   d_{\bar\imath+\bar\jmath}$ where $F$ is any field of
    characteristic $0$.
Then there exist $C_1, C_2 > 0$ and $r_1, r_2 \in \mathbb R$ such that
$$C_1 n^{r_1} m^n \leqslant c_n^{\mathrm{gr}}(L) \leqslant C_2 n^{r_2} m^n.$$
\end{example}
\begin{proof}
Again, we may assume $F$ to be
algebraically closed.
Let $\zeta \in F$ be an $m$th primitive root of $1$.
Then $\widehat G
=\lbrace \psi_0, \ldots, \psi_{m-1}\rbrace$
 for $G=\mathbb Z_m$
where $\psi_\ell(\bar\jmath):=\zeta^{\ell j}$.
 We can identify the algebras from Example~\ref{ExampleMetabelian}
 and Example~\ref{ExampleMetabelianGr}
by formulas $c_{\bar\jmath} = \sum_{k=1}^m \zeta^{-jk} a_k$ and
$d_{\bar\jmath} = \sum_{k=1}^m \zeta^{-jk} b_k$. The
$\mathbb Z_m$-grading and $\langle\tau\rangle$-action
correspond to each other since $(c_{\bar\jmath})^{\tau^\ell} = \zeta^{\ell j} c_{\bar\jmath}
=\psi_\ell(\bar\jmath) c_{\bar\jmath}$
and $(d_{\bar\jmath})^{\tau^\ell} = \zeta^{\ell j} d_{\bar\jmath}=\psi_\ell(\bar\jmath)
 d_{\bar\jmath}$.
By Lemma~\ref{LemmaGradAction}, $c_n^{\mathrm{gr}}(L)=c_n^{\langle\tau\rangle}(L)$
and  the bounds follow from~Example~\ref{ExampleMetabelian}.
\end{proof}

\subsection{$S_n$-cocharacters}
One of the main tools in the investigation of polynomial
identities is provided by the representation theory of symmetric groups.
 The symmetric group $S_n$ acts
 on the space $\frac {V^G_n}{V^G_n \cap \Id^G(L)}$
  by permuting the variables.
  Irreducible $FS_n$-modules are described by partitions
  $\lambda=(\lambda_1, \ldots, \lambda_s)\vdash n$ and their
  Young diagrams $D_\lambda$.
   The character $\chi^G_n(L)$ of the
  $FS_n$-module $\frac {V^G_n}{V^G_n \cap \Id^G(L)}$ is called the $n$th
  \textit{cocharacter} of polynomial $G$-identities of $L$.
  We can rewrite it as
  a sum $\chi^G_n(L)=\sum_{\lambda\vdash n} m(L, G, \lambda) \chi(\lambda)$ of
  irreducible characters $\chi(\lambda)$.
Let  $e_{T_{\lambda}}=a_{T_{\lambda}} b_{T_{\lambda}}$
and
$e^{*}_{T_{\lambda}}=b_{T_{\lambda}} a_{T_{\lambda}}$
where
$a_{T_{\lambda}} = \sum_{\pi \in R_{T_\lambda}} \pi$
and
$b_{T_{\lambda}} = \sum_{\sigma \in C_{T_\lambda}} (\sign \sigma) \sigma$,
be the Young symmetrizers corresponding to a Young tableau~$T_\lambda$.
Then $M(\lambda) = FS e_{T_\lambda} \cong FS e^{*}_{T_\lambda}$
is an irreducible $FS_n$-module corresponding to the partition~$\lambda \vdash n$.
  We refer the reader to~\cite{ZaiGia, Bahturin, DrenKurs} for an account
  of $S_n$-representations and their applications to polynomial
  identities.

\medskip

Our proof of Theorem~\ref{TheoremMainG} follows the outline of the proof by M.V.~Zaicev~\cite{ZaiLie}.
However, in many cases we need to apply new ideas.

In Section~\ref{SectionAux} we discuss modules with
$G$-action over Lie $G$-algebras, their
annihilators and complete reducibility.

In Section~\ref{SectionMult} we prove that $m(L, G, \lambda)$
is polynomially bounded. In Section~\ref{SectionUpper}
we prove that if $m(L, G, \lambda) \ne 0$, then
the corresponding Young diagram $D_\lambda$
has at most $d$ long rows. This implies the upper bound.

In Section~\ref{SectionAlt} we
consider faithful irreducible $L_0$-modules with $G$-action
where $L_0$ is a reductive Lie $G$-algebra.
For an arbitrary $k\in\mathbb N$,
we construct an associative $G$-polynomial that is
alternating in $2k$ sets,
 each consisting of $\dim L_0$ variables. This polynomial
 is not an identity of the corresponding representation
 of $L_0$. In Section~\ref{SectionLower} we choose reductive
 algebras and faithful irreducible modules with $G$-action,
 and glue the corresponding alternating polynomials.
 This allows us to find $\lambda \vdash n$
 with $m(L, G, \lambda)\ne 0$ such that $\dim M(\lambda)$
 has the desired asymptotic behavior and the lower bound is proved.

\section{Lie algebras and modules with $G$-action}
\label{SectionAux}

We need several auxiliary lemmas. First, the Weyl
theorem~\cite[Theorem 6.3]{Humphreys} on complete
 reducibility of representations can be easily extended
to the case of Lie algebras with $G$-action.

\begin{lemma}\label{LemmaComplIrrGInv} Let $M$ be a finite dimensional
module with $G$-action over a Lie $G$-algebra $L_0$. Suppose $M$ is a
completely reducible $L_0$-module
disregarding the $G$-action. Then
$M$ is completely reducible $L_0$-module with $G$-action.
\end{lemma}
\begin{corollary}
If $M$ is a finite dimensional module with $G$-action over a semisimple
Lie $G$-algebra $B_0$, then $M$ is a completely reducible module with $G$-action.
\end{corollary}
\begin{proof}[Proof of Lemma~\ref{LemmaComplIrrGInv}]
 Suppose $M_1 \subseteq M$ is a $G$-invariant $L_0$-submodule of $M$.
Then it is sufficient
to prove that there exists a $G$-invariant $L_0$-submodule
 $M_2 \subseteq M$
such that $M =M_1 \oplus M_2$.

Since $M$ is completely reducible, there exists an $L_0$-homomorphism $\pi \colon M \to M_1$
such that $\pi(v)=v$ for all $v \in M_1$. Consider a homomorphism
$\tilde \pi \colon M \to M_1$,
 $\tilde\pi(v)=\frac{1}{|G|}\sum_{g\in G} \pi(v^{g^{-1}})^g$.
 Then $\tilde\pi(v)=v$ for all $v \in M_1$ too and for all
 $a\in L_0$, $h \in G$ we have
 $$\tilde\pi(a\cdot v)=\frac{1}{|G|}\sum_{g\in G} \pi((a \cdot v)^{g^{-1}})^g
 =\frac{1}{|G|}\sum_{g\in G} \pi(a^{g^{-1}} \cdot v^{g^{-1}})^g
 =\frac{1}{|G|}\sum_{g\in G} a \cdot \pi(v^{g^{-1}})^g
 = a\cdot \tilde\pi(v),$$
 $$\tilde\pi(v^h)=\frac{1}{|G|}\sum_{g\in G} \pi((v^h)^{g^{-1}})^g
 =\frac{1}{|G|}\sum_{g\in G} \pi(v^{(h^{-1}g)^{-1}})^{h(h^{-1}g)}
 =\frac{1}{|G|}\sum_{g'\in G} (\pi(v^{{g'}^{-1}})^{g'})^h
 = \tilde\pi(v)^h$$
 where $g'=h^{-1}g$.
 Thus we can take $M_2=\ker \tilde \pi$.
\end{proof}

 Note that $[L, R] \subseteq N$
by \cite[Proposition 2.1.7]{GotoGrosshans} where $N$ is the nilpotent radical,
which is a $G$-invariant ideal.

\begin{lemma}\label{LemmaRS}
There exists a $G$-invariant subspace $S \subseteq R$ such that
 $R=S\oplus N$ is the direct sum of subspaces and $[B, S]=0$.
\end{lemma}
\begin{proof}
Note that $R$ is a $B$-submodule under the adjoint representation
of $B$ on $L$.
Applying the corollary of Lemma~\ref{LemmaComplIrrGInv} to $N \subseteq R$,
we obtain a $G$-invariant complementary subspace $S \subseteq R$
such that $[B, S] \subseteq S$.
Thus $[B, S] \subseteq S \cap [L, R] \subseteq S \cap N = 0$.
\end{proof}

Therefore, $L=B\oplus S\oplus N$ (direct sum of subspaces).

Let $M$ be an $L$-module and  let $T$ be a subspace of $L$.
Denote $\Ann_T M := (\Ann M) \cap T$.
Lemma~\ref{LemmaIrrAnnBS} is a $G$-invariant analog of~\cite[Lemma 4]{ZaiLie}.

\begin{lemma} \label{LemmaIrrAnnBS}
Let $J \subseteq I \subseteq L$ be $G$-invariant ideals
such that  $I/J$ is an irreducible $L$-module with $G$-action.
Then \begin{enumerate}
\item $\Ann_B (I/J)$ and $\Ann_S (I/J)$ are $G$-invariant subspaces
of $L$; \label{BSInv}
\item $\Ann (I/J)=\Ann_B (I/J)\oplus \Ann_S (I/J)
\oplus N$. \label{BSDecomp}
\end{enumerate}
\end{lemma}
\begin{proof}
Since $I/J$ is a module with $G$-action,
$\Ann (I/J)$,
$\Ann_B (I/J)$, and $\Ann_S (I/J)$ are $G$-invariant.
Moreover $[N, I] \subseteq J$ since $N$ is a nilpotent ideal
 and $I/J$ is a composition factor of the adjoint representation.
Hence $N \subseteq \Ann(I/J)$. In order to prove the lemma, it is sufficient to show that if $b+s \in \Ann (I/J)$, $b \in B$, $s \in S$, then
$b,s \in \Ann (I/J)$.
Denote
$\varphi \colon L \to \mathfrak{gl}(I/J)$. Then $\varphi(b)+\varphi(s)=0$
and $$[\varphi(b), \varphi(B)]=[-\varphi(s), \varphi(B)]=0.$$
Hence $\varphi(b)$ belongs to the center of $\varphi(B)$ and $\varphi(b)=\varphi(s)=0$ since $\varphi(B)$ is semisimple.
Thus $b,s \in \Ann (I/J)$ and the lemma is proved.
\end{proof}

\begin{lemma} \label{LemmaRedIrr}
Let $L_0=B_0 \oplus R_0$ be a finite dimensional reductive Lie algebra
with $G$-action, $B_0$ be a maximal semisimple $G$-subalgebra, and
$R_0$ be the center of $L_0$.
Let $M$ be a finite dimensional irreducible $L_0$-module with $G$-action.
Then \begin{enumerate}
\item \label{RedSum} $M=M_1 \oplus \ldots \oplus M_q$ for some
$L_0$-submodules $M_i$, $1 \leqslant i \leqslant q$;
\item \label{RedScalar} elements of $R_0$ act on each $M_i$ by scalar operators;
\item for every $1 \leqslant i \leqslant q$ and $g\in G$ there exists
such $1 \leqslant j \leqslant q$ that $M_i^g = M_j$
and this action of $G$ on the set $\lbrace M_1, \ldots, M_q \rbrace$ is transitive. \label{RedGAction}
\end{enumerate}
\end{lemma}
\begin{proof}
Denote by $\varphi$ the homomorphism $L_0 \to \mathfrak{gl}(M)$.
Then $\varphi$ is a homomorphism of $G$-representations.
 We claim that $\varphi(R_0)$ consist of semisimple
operators.
Let $r_1, \ldots, r_t$ be a basis in $R_0$. Consider the Jordan decomposition $\varphi(r_i)=r'_i+r''_i$ where each $r'_i$ is semisimple, each $r''_i$ is nilpotent,
and both are polynomials of $\varphi(r_i)$ without a constant term \cite[Section 4.2]{Humphreys}. Since each $\varphi(r_i)$ commutes with all operators $\varphi(a)$, $a\in L_0$, the elements $(r''_i)^g$, $1\leqslant i \leqslant t$, $g \in G$, generate a nilpotent
$G$-invariant associative ideal $K$ in the enveloping algebra $A \subseteq \End_F(M)$ of the Lie algebra $\varphi(L_0)$.
Suppose $KM \ne 0$. Then for some $\varkappa \in \mathbb N$
we have $K^{\varkappa+1}M=0$, but $K^{\varkappa}M\ne 0$.
Note that $K^{\varkappa}M$ is a non-zero $G$-invariant
$L_0$-submodule. Thus $K^{\varkappa}M=M$
and $KM=K^{\varkappa+1}M=0$. Since $K \subseteq \End_F(M)$,
we obtain $K=0$.

Therefore $\varphi(r_i)=r'_i$ are commuting semisimple operators.
They have a common basis of eigenvectors. Hence we can choose
subspaces $M_i$, $1 \leqslant i \leqslant q$, $q\in \mathbb N$,
such that $$M=M_1 \oplus \ldots \oplus M_q,$$ and each $M_i$
is the intersection of eigenspaces of $\varphi(r_i)$.
Note that $[\varphi(r_i),\varphi(x)]=0$ for all $x\in L_0$.
Thus $M_i$ are $L_0$-submodules and Propositions~\ref{RedSum} and \ref{RedScalar} are proved.

For every $M_i$ we can define a linear function $\alpha_i \colon R_0 \to F$
such that $\varphi(r)m=\alpha_i(r)m$ for all $r \in R_0$ and $m \in M_i$.
Then $M_i = \bigcap_{r\in R_0} \ker(\varphi(r)-\alpha_i(r)\cdot 1)$
and $$M_i^g = \bigcap_{r\in R_0} \ker(\varphi(r^g)-\alpha_i(r)\cdot 1)
= \bigcap_{\tilde r\in R_0}
 \ker(\varphi(\tilde r)-\alpha_i(\tilde r^{g^{-1}})\cdot 1)$$
 where $\tilde r=r^g$.
Therefore, $M_i^g$ must coincide with $M_j$ for some $1 \leqslant j \leqslant q$.
The module $M$ is irreducible with respect to $L_0$- and $G$-action that implies Proposition~\ref{RedGAction}.
\end{proof}

\begin{lemma}\label{LemmaLR} Let $W$ be a finite dimensional
$L$-module with $G$-action. Let $\varphi \colon L \to \mathfrak{gl}(W)$
 be the corresponding homomorphism.  Denote by $A$ the associative
subalgebra of $\End_F(W)$ generated by the operators from $\varphi(L)$
and $G$. Then
$\varphi([L, R]) \subseteq J(A)$ where $J(A)$ is the Jacobson radical of $A$.
\end{lemma}
% See also Bourbaki N. Lie groups an algebras. Chapter 1, Section 5, Theorem 1.
\begin{proof}
Let $W = W_0 \supseteq W_1 \supseteq W_2 \supseteq \ldots \supseteq W_t = \left\lbrace 0 \right\rbrace$
be a composition chain in $W$ of not necessarily $G$-invariant
$L$-submodules.
Then each $W_i/W_{i+1}$ is an irreducible $L$-module.
Denote the corresponding homomorphism by $\varphi_i \colon L \to
\mathfrak{gl}(W_i/W_{i+1})$. Then by E.~Cartan's theorem~\cite[Proposition~1.4.11]{GotoGrosshans},
$\varphi_i (L)$ is semisimple or the direct sum of a semisimple ideal and
the center of $\mathfrak{gl}(W_i/W_{i+1})$.
Thus $\varphi_i ([L, L])$ is semisimple
and $\varphi_i ([ L, L ] \cap R) = 0$.
Since $[L, R] \subseteq [L, L] \cap R$,
we have $\varphi_i([L, R])=0$ and $[L,R]W_i \subseteq W_{i+1}$.
Denote by $\rho \colon G \to \mathrm{GL}(W)$ the homomorphism
corresponding to $G$-action.
The associative $G$-invariant ideal of $A$
 generated by $\varphi([L, R])$ is nilpotent since
 for any $a_i \in \varphi([L, R])$, $b_{ij}\in\varphi(L)$, $g_{ij}\in G$
 we have
$$
a_1 \bigl(\rho(g_{10})b_{11}\rho(g_{11}) \ldots \rho(g_{1,s_1-1})b_{1,s_1}
\rho(g_{1,s_1})\bigr) a_2 \ldots
$$ $$
a_{t-1}\bigl(\rho(g_{t-1,0})b_{t-1,1}\rho(g_{t-1,1}) \ldots
 \rho(g_{t-1,s_{t-1}-1})b_{t-1,s_{t-1}}\rho(g_{t-1,s_{t-1}})\bigr) a_t =$$ $$
a_1 \bigl(b_{11}^{g_{10}} \ldots b_{1,s_1}^{g'_{1,s_1}}
\bigr) a_2^{g_2} \ldots
a_{t-1}^{g_{t-1}}
\bigl(b_{t-1,1}^{g'_{t-1,1}} \ldots
 b_{t-1,s_{t-1}}^{g'_{t-1,s_{t-1}}}\bigr) a_t^{g_t} \rho(g_{t+1}) = 0
$$ where $g_i, g'_{ij} \in G$ are products of $g_{ij}$ obtained
using the property $\rho(g)bw = b^g \rho(g)w$ where
$g\in G$, $b\in \varphi(L)$, $w \in W$.
Thus $\varphi([L, R]) \subseteq J(A)$.
\end{proof}

\section{Multiplicities of irreducible characters in $\chi^G_n(L)$}
\label{SectionMult}

The aim of the section is to prove

\begin{theorem}\label{TheoremMult}
Let $L$ be a finite dimensional Lie $G$-algebra
over a field $F$ of characteristic $0$ where $G$ is a finite group. Then there exist constants $C > 0$, $r \in \mathbb N$ such that
 $$\sum_{\lambda \vdash n} m(L,G,\lambda)
\leqslant Cn^r$$ for all $n \in \mathbb N$.
\end{theorem}

\begin{remark}
Cocharacters do not change
upon an extension of the base field $F$
(the proof is completely analogous to \cite[Theorem 4.1.9]{ZaiGia}),
 so we may assume $F$ to be algebraically closed.
\end{remark}

In~\cite[Theorem 13 (b)]{BereleHopf} A. Berele, using
the duality between $S_n$- and
 $\mathrm{GL}_m(F)$-cocharacters~\cite{DrenskySb,BereleHom}, showed that
such sequence for an associative algebra
with an action of a Hopf algebra is polynomially bounded.
One may repeat those steps for Lie $G$-algebras and prove Theorem~\ref{TheoremMult}
in that way.
However we provide an alternative proof based only on $S_n$-characters.

Let $\lbrace e \rbrace$ be the trivial group,
$V_n := V^{\lbrace e \rbrace}_n$,
$\chi_n(L) := \chi^{\lbrace e \rbrace}_n(L)$,
$m(L,\lambda):= m(L,\lbrace e \rbrace,\lambda)$,
$\Id(L):=\Id^{\lbrace e \rbrace}(L)$.
Then, by~\cite[Theorem~3.1]{GiaRegZaicev},
\begin{equation}\label{EqOrdinaryColength}
\sum_{\lambda \vdash n} m(L,\lambda)
\leqslant C_3 n^{r_3}
\end{equation}
 for some $C_3 > 0$ and $r_3 \in \mathbb N$.

Let $G_1 \subseteq G_2$ be finite groups and $W_1$, $W_2$ be
$FG_1$- and $FG_2$-modules respectively.
Then we denote $FG_2$-module
$FG_2 \otimes_{FG_1} W_1$
 by $W_1 \uparrow G_2$.
 Here $G_2$ acts on the first component.
Let $W_2 \downarrow G_1$ be $W_2$ with $G_2$-action restricted to $G_1$.
We use analogous notation for the characters.

Denote by  $\length(M)$ the number of irreducible components
of a module $M$.

Consider the diagonal embedding $\varphi \colon S_n \to S_{n|G|}$,
$$\varphi(\sigma):=\left(\begin{array}{cccc|lllc|c}
1 & 2 & \ldots & n & n+1 & n+2 & \ldots & 2n & \ldots  \\
\sigma(1) & \sigma(2) & \ldots & \sigma(n) &
n+\sigma(1) & n+\sigma(2) & \ldots & n+\sigma(n) & \ldots
\end{array}
\right).$$

 Then we have

\begin{lemma}\label{LemmaEmb}
$$ \sum_{\lambda \vdash n} m(L,G,\lambda)
= \length\left(\frac{V^G_n}{V^G_n \cap \Id^G(L)}\right)
 \leqslant \length\left(\left(\frac{V_{n|G|}}{V_{n|G|} \cap \Id(L)}\right)\downarrow \varphi(S_n)\right).
$$
\end{lemma}
\begin{proof}
Consider $S_n$-isomorphism $\pi \colon (V_{n|G|}\downarrow \varphi(S_n)) \to V_n^G$
 defined by $\pi(x_{n(i-1)+t})=x^{g_i}_t$
 where $G=\lbrace g_1, g_2, \ldots, g_{|G|} \rbrace$,
 $1 \leqslant t \leqslant n$. Note that $\pi(V_{n|G|} \cap \Id(L))
 \subseteq V^G_n \cap \Id^G(L)$.
 Thus $FS_n$-module $\frac{V^G_n}{V^G_n \cap \Id^G(L)}$
 is a homomorphic image of $FS_n$-module
 $\left(\frac{V_{n|G|}}{V_{n|G|} \cap \Id(L)}\right)\downarrow \varphi(S_n)$.
 \end{proof}

 Hence it is sufficient to prove
 that  $\length\left(\left(\frac{V_{n|G|}}{V_{n|G|} \cap \Id(L)}\right)\downarrow \varphi(S_n)
 \right)$ is polynomially bounded.
However, we start with the study of the restriction on the larger subgroup $$S\{1,\ldots,n\}  \times S\{n+1,\ldots,2n\} \times \ldots \times S\{n(|G|-1),\ldots,n|G|\}
 \subseteq S_{n|G|}$$ that we denote by $(S_n)^{|G|}$.

This is a particular case of a more general situation.
Let $m=m_1 + \ldots + m_t$, $m_i \in \mathbb N$.
 Then we have a natural embedding $S_{m_1} \times \ldots \times S_{m_t}
\hookrightarrow S_m$. Irreducible representations of
$S_{m_1} \times \ldots \times S_{m_t}$
 are isomorphic to $M(\lambda^{(1)})\sharp
 \ldots \sharp M(\lambda^{(t)})$ where $\lambda^{(i)}\vdash m_i$.
 Here $$M(\lambda^{(1)})\sharp
 \ldots \sharp M(\lambda^{(t)}) \cong M(\lambda^{(1)})\otimes
 \ldots \otimes M(\lambda^{(t)})$$ as a vector space and
 $S_{m_i}$ acts on $M(\lambda^{(i)})$.
 Denote by $\chi(\lambda^{(1)}) \sharp
 \ldots \sharp \chi(\lambda^{(t)})$ the character of $M(\lambda^{(1)})\sharp
 \ldots \sharp M(\lambda^{(t)})$.

 Analogously, $\chi(\lambda^{(1)}) \hatotimes
 \ldots \hatotimes \chi(\lambda^{(t)})$
 is the character of $FS_m$-module
$$ M(\lambda^{(1)})\hatotimes
 \ldots \hatotimes M(\lambda^{(t)})
 := (M(\lambda^{(1)})\sharp
 \ldots \sharp M(\lambda^{(t)})) \uparrow S_m.$$

Note that if $m_1 = \ldots = m_t = k$, one can define
the {\itshape inner tensor product}, i.e.
$$M(\lambda^{(1)})\otimes
 \ldots \otimes M(\lambda^{(t)})$$
 with the diagonal $S_k$-action.
 The character of this $FS_k$-module equals $\chi(\lambda^{(1)})
\ldots \chi(\lambda^{(t)})$.

 Recall that irreducible characters of any
finite group $G_0$ are orthonormal with respect
to the scalar product
$(\chi, \psi)=\frac{1}{|G_0|}\sum_{g\in G_0} \chi(g^{-1})\psi(g)$.

Denote by $\lambda^T$ the transpose partition of $\lambda \vdash n$.
Then $\lambda^T_1$ equals the height of the first column of $D_\lambda$.

 \begin{lemma}\label{LemmaTensorRestrict}
 Let $h, t \in \mathbb N$.
 There exist $C_4 > 0$, $r_4 \in \mathbb N$
 such that for all $\lambda \vdash m$, $\lambda^{(1)} \vdash m_1$, \ldots, $\lambda^{(t)} \vdash m_t$, where $D_\lambda$ lie in the strip of height $h$, i.e. $\lambda^T_1
 \leqslant h$, and $m_1+m_2+\ldots + m_t=m$, we have
 $$\Bigl(\chi(\lambda) \downarrow (S_{m_1} \times \ldots \times S_{m_t}),
 \ \chi(\lambda^{(1)}) \sharp
 \ldots \sharp \chi(\lambda^{(t)})\Bigr) =
 \Bigl(\chi(\lambda),\ \chi(\lambda^{(1)}) \hatotimes
 \ldots \hatotimes \chi(\lambda^{(t)}) \Bigr)
 \leqslant  C_4 m^{r_4}.$$ If
 $\lambda \vdash m$, $\lambda^{(1)} \vdash m_1$, \ldots, $\lambda^{(t)} \vdash m_t$, $m_1+m_2+\ldots+ m_t=m$, and
 $$\Bigl(\chi(\lambda) \downarrow (S_{m_1} \times \ldots \times S_{m_t}),\ \chi(\lambda^{(1)}) \sharp
 \ldots \sharp \chi(\lambda^{(t)})\Bigr)=
 \Bigl(\chi(\lambda),\ \chi(\lambda^{(1)}) \hatotimes
 \ldots \hatotimes \chi(\lambda^{(t)}) \Bigr)\ne 0$$
 then
 $\left(\lambda^{(i)}\right)^T_1 \leqslant \lambda^T_1$ for all $1 \leqslant i \leqslant t$
 and $\lambda^T_1 \leqslant \sum_{i=1}^t \left(\lambda^{(i)}\right)^T_1$.
 \end{lemma}
 \begin{proof}
 By Frobenius reciprocity, $$\Bigl(\chi(\lambda) \downarrow (S_{m_1} \times \ldots \times S_{m_t}),\ \chi(\lambda^{(1)}) \sharp
 \ldots \sharp \chi(\lambda^{(t)}) \Bigr)=
 \Bigl(\chi(\lambda),\ (\chi(\lambda^{(1)}) \sharp
 \ldots \sharp \chi(\lambda^{(t)}))\uparrow S_m\Bigr)=$$
$$ \Bigl(\chi(\lambda),\ \chi(\lambda^{(1)}) \hatotimes
 \ldots \hatotimes \chi(\lambda^{(t)}) \Bigr).$$

Now we prove the lemma by induction on $t$.
The case $t=1$ is trivial.
Suppose $\Bigl(\chi(\mu),\ \chi(\lambda^{(1)}) \hatotimes
 \ldots \hatotimes \chi(\lambda^{(t-1)}) \Bigr)$
 is polynomially bounded for every $\mu \vdash (m_1+\ldots+m_{t-1})$
 with $\mu^T_1 \leqslant h$.
 We have
 $$\Bigl(\chi(\lambda),\ \chi(\lambda^{(1)}) \hatotimes
 \ldots \hatotimes \chi(\lambda^{(t)}) \Bigr)
 =\Bigl(\chi(\lambda),\ \left(\chi(\lambda^{(1)}) \hatotimes
 \ldots \hatotimes \chi(\lambda^{(t-1)})\right)\hatotimes \chi(\lambda^{(t)}) \Bigr)
 =$$
 \begin{equation}\label{EqLittlewood}
\sum_{\mu \vdash (m_1+\ldots+m_{t-1})}
 \Bigl(\chi(\mu),\ \chi(\lambda^{(1)}) \hatotimes
 \ldots \hatotimes \chi(\lambda^{(t-1)}) \Bigr)
  \Bigl(\chi(\lambda), \chi(\mu)\hatotimes\chi(\lambda^{(t)})\Bigr).
   \end{equation}
 In order to determine
 the multiplicity
 of $\chi(\lambda)$
 in $\chi(\mu) \hatotimes
  \chi(\lambda^{(t)})$,
  we are using the Littlewood~--- Richardson rule
(see the algorithm in~\cite[Corollary~2.8.14]{JamesKerber}).
We cannot obtain $D_\lambda$
if $\left(\lambda^{(t)}\right)^T_1 > \lambda^T_1$ or $\mu^T_1 > \lambda^T_1$,
or $\lambda^T_1 > \left(\lambda^{(t)}\right)^T_1 + \mu^T_1$.
 Suppose the Young diagram
 $D_{\lambda}$
lies in the strip of height $h$. Then
we may consider only the case $\left(\lambda^{(t)}\right)^T_1 \leqslant h$
and  $\mu^T_1 \leqslant h$.
Each time the number of variants to add the boxes
from a row is bounded by $m^{h}$.
Since $\left(\lambda^{(t)}\right)^T_1 \leqslant h$, the second multiplier in~(\ref{EqLittlewood}) is
bounded by $(m^h)^h=m^{h^2}$.
The number of diagrams in the strip of height $h$
is bounded by $m^h$.
Thus the number of
 terms in~(\ref{EqLittlewood}) is bounded
 by $m^h$. Together with the inductive assumption this yields the lemma.
 \end{proof}

  \begin{lemma}\label{LemmaTensorColength}
  There exist $C_5 > 0$, $r_5 \in \mathbb N$ such that
  $$\length\left(\left(\frac{V_{n|G|}}{V_{n|G|} \cap \Id(L)}\right)\downarrow (S_n)^{|G|}
 \right) \leqslant C_5 n^{r_5}$$ for all $n \in \mathbb N$.
 Moreover, if $\left(\lambda^{(i)}\right)^T_1 >
 \dim L$ for some $1 \leqslant i \leqslant |G|$,
 then
 $M(\lambda^{(1)})\sharp
 \ldots \sharp M(\lambda^{(|G|)})$ does not appear in the decomposition.
 \end{lemma}
\begin{proof}
Fix an $|G|$-tuple
of partitions $(\lambda^{(1)}, \ldots, \lambda^{(|G|)})$,
$\lambda^{(i)} \vdash n$.
Then the multiplicity of $M(\lambda^{(1)})\sharp
 \ldots \sharp M(\lambda^{(|G|)})$ in  $\left(\frac{V_{n|G|}}{V_{n|G|} \cap \Id(L)}\right)\downarrow (S_n)^{|G|}$
equals $$\Bigl(\chi(\lambda^{(1)}) \sharp
 \ldots \sharp \chi(\lambda^{(|G|)}),
\ \chi_{n|G|}(L) \downarrow  (S_n)^{|G|}\Bigr)=$$
\begin{equation}\label{EqMultTensor}
\sum_{\lambda \vdash n|G|}\Bigl(\chi(\lambda^{(1)}) \sharp
 \ldots \sharp \chi(\lambda^{(|G|)}),
 \ \chi(\lambda)\downarrow  (S_n)^{|G|}\Bigr) m(L,\lambda).
\end{equation}

By~\cite[Lemma~3.4]{GiaRegZaicev} (or Lemma~\ref{LemmaUpper} for $G=\langle e \rangle$),
 $m(L, \lambda)=0$ for all $\lambda \vdash n|G|$ with
 $\lambda^T_1 > \dim L$.
Thus Lemma~\ref{LemmaTensorRestrict}
implies that
for all $M(\lambda^{(1)})\sharp
 \ldots \sharp M(\lambda^{(|G|)})$ that appear in $\left(\frac{V_{n|G|}}{V_{n|G|} \cap \Id(L)}\right)\downarrow (S_n)^{|G|}$, the
 Young diagrams $D_{\lambda^{(i)}}$
lie in the strip of height $(\dim L)$.
 Thus the number of different $(\lambda^{(1)}, \ldots, \lambda^{(|G|)})$
that appear in the decomposition of $\left(\frac{V_{n|G|}}{V_{n|G|} \cap \Id(L)}\right)\downarrow (S_n)^{|G|}$ is bounded by $n^{(\dim L)|G|}$.
Together with~(\ref{EqOrdinaryColength}), (\ref{EqMultTensor}),
and Lemma~\ref{LemmaTensorRestrict}, this yields the lemma.
\end{proof}

\begin{lemma}\label{LemmaInnerTensor}
Let $h, k \in \mathbb N$.
There exist $C_6 > 0$, $r_6 \in \mathbb N$ such that
for the inner tensor product $M(\lambda) \otimes M(\mu)$
 of any $FS_n$-modules $M(\lambda)$
and $M(\mu)$, $\lambda, \mu \vdash n$, $\lambda^T_1\leqslant h$, $\mu^T_1 \leqslant k$,
we have $$\length_{S_n}(M(\lambda) \otimes M(\mu)) \leqslant C_6 n^{r_6}$$
and $\Bigl(\chi(\lambda)\chi(\mu), \chi(\nu)\Bigr)=0$ for any $\nu \vdash n$
with $\nu^T_1 > hk$.
\end{lemma}
\begin{proof} Let $T_\mu$ be any Young tableau of the shape $\mu$.
Denote by $IR_{T_\mu}$ the one-dimensional
trivial representation of the Young subgroup (i.e. the row stabilizer)
$R_{T_\mu}$. Then $$FS_n a_{T_\mu} \cong IR_{T_\mu}
\uparrow S_n$$ (see \cite[Section 4.3]{FultonHarris}).
 By \cite[Theorem~38.5]{CurtisReiner},
 $$M(\lambda) \otimes (IR_{T_\mu} \uparrow S_n)
\cong ((M(\lambda)\downarrow R_{T_\mu}) \otimes IR_{T_\mu}) \uparrow S_n.$$
Thus
$$M(\lambda) \otimes M(\mu) \cong M(\lambda) \otimes FS_n e^{*}_{T_\mu}=
M(\lambda) \otimes FS_n b_{T_\mu} a_{T_\mu}
\subseteq M(\lambda) \otimes FS_n a_{T_\mu}
\cong $$ $$M(\lambda) \otimes (IR_{T_\mu} \uparrow S_n)
\cong ((M(\lambda)\downarrow R_{T_\mu}) \otimes IR_{T_\mu}) \uparrow S_n
\cong (M(\lambda)\downarrow R_{T_\mu}) \uparrow S_n.$$
Note that $\length(M(\lambda)\downarrow R_{T_\mu})$
is polynomially bounded by Lemma~\ref{LemmaTensorRestrict}
and $M(\lambda)\downarrow R_{T_\mu}$
is a sum of $M(\varkappa^{(1)}) \sharp \ldots \sharp
M(\varkappa^{(s)})$, $s=\mu^T_1 \leqslant k$, $\varkappa^{(i)}\vdash \mu_i$,
$\left(\varkappa^{(i)}\right)^T_1 \leqslant h$.
Thus
$(M(\lambda)\downarrow R_{T_\mu}) \uparrow S_n$
is a sum of $M(\varkappa^{(1)}) \hatotimes \ldots \hatotimes
M(\varkappa^{(s)})$. Applying Lemma~\ref{LemmaTensorRestrict}
again, we obtain the lemma.
\end{proof}

\begin{lemma}\label{LemmaIrrTensor}
  There exist $C_7 > 0$, $r_7 \in \mathbb N$ satisfying
  the following properties. If
  $(\lambda^{(1)}, \ldots, \lambda^{(|G|)})$
  is an  $|G|$-tuple
 of partitions $\lambda^{(i)} \vdash n$
where $\left(\lambda^{(i)}\right)^T_1 \leqslant \dim L$ for all $1 \leqslant i \leqslant |G|$, then
  $$\length_{S_n}\left(M(\lambda^{(1)})\otimes
 \ldots \otimes M(\lambda^{(|G|)})
 \right) \leqslant C_7 n^{r_7}.$$
\end{lemma}
\begin{proof}
Note that $$M(\lambda^{(1)})\otimes
 \ldots \otimes M(\lambda^{(t)})
 = (M(\lambda^{(1)})\otimes
 \ldots \otimes M(\lambda^{(t-1)}))\otimes M(\lambda^{(t)}).$$
 Using induction on $t$ and applying
Lemma~\ref{LemmaInnerTensor} with $h=(\dim L)^{t-1}$ and $k=\dim L$,
we obtain the lemma.
\end{proof}

 \begin{proof}[Proof of Theorem~\ref{TheoremMult}]
The theorem is an immediate
consequence of Lemmas~\ref{LemmaEmb}, \ref{LemmaTensorColength}, and \ref{LemmaIrrTensor}.
\end{proof}

\section{Upper bound}
\label{SectionUpper}

Fix a composition chain of $G$-invariant ideals
$$L=L_0 \supsetneqq L_1 \supsetneqq L_2 \supsetneqq \ldots \supsetneqq
N\supsetneqq \ldots \supsetneqq L_{\theta-1}
 \supsetneqq L_\theta = \{0\}.$$
 Let $\height a := \max_{a \in L_k} k$ for $a \in L$.

\begin{remark}
If $d=d(L)=0$, then $L = \Ann(L_{i-1}/L_i)$
for all $1 \leqslant i \leqslant \theta$ and
 $[a_1, a_2, \ldots, a_n] =0$ for all $a_i \in L$
 and $n \geqslant \theta +1$. Thus $c^G_n(L)=0$
 for all $n \geqslant \theta +1$. Therefore we assume $d > 0$.
\end{remark}

 Let $Y:=\lbrace y_{11}, y_{12}, \ldots, y_{1j_1};\,
 y_{21}, y_{22}, \ldots, y_{2j_2}; \ldots;\,
 y_{m1}, y_{m2}, \ldots, y_{mj_m}\rbrace$,
 $Y_1$, \ldots, $Y_q$, and $\lbrace z_1, \ldots, z_m\rbrace$
 be subsets of $\lbrace x_1, x_2, \ldots, x_n\rbrace$
 such that $Y_i \subseteq Y$, $|Y_i|=d+1$, $ Y_i \cap Y_j = \varnothing$
 for $i \ne j$,
 $Y \cap \lbrace z_1, \ldots, z_m\rbrace = \varnothing$,
  $j_i \geqslant 0$.
  Denote $$f_{m,q}:=\Alt_{1} \ldots \Alt_{q} [[z_1^{g_1}, y_{11}^{g_{11}}, y_{12}^{g_{12}},
  \ldots, y_{1j_1}^{g_{1j_1}}],
 [z_2^{g_2},y_{21}^{g_{21}},y_{22}^{g_{22}},\ldots, y_{2j_2}^{g_{2j_2}}], \ldots,
 $$ $$[ z_m^{g_m}, y_{m1}^{g_{m1}}, y_{m2}^{g_{m2}}, \ldots, y_{mj_m}^{g_{mj_m}}]]$$
 where $\Alt_i$ is the operator of alternation on the variables of $Y_i$,
 $g_i, g_{ij}\in G$.

Let $\varphi \colon L(X | G) \to L$
be a $G$-homomorphism induced by some substitution $\lbrace x_1, x_2, \ldots, x_n \rbrace \to L$.
We say that $\varphi$ is \textit{proper} for  $f_{m,q}$ if
 $\varphi(z_1) \in N \cup B \cup S$,
 $\varphi(z_i) \in N$ for $2\leqslant i \leqslant m$,
  and
  $\varphi(y_{ik})\in B \cup S$ for $1\leqslant i \leqslant m$,
   $1 \leqslant k \leqslant j_i$.

\begin{lemma}\label{LemmaReduct}
Let $\varphi$ be a \textit{proper} homomorphism for $f_{m,q}$.
Then $\varphi(f_{m,q})$ can be rewritten as a
sum of $\psi(f_{m+1,q'})$ where $\psi$
is a proper homomorphism for $f_{m+1,q'}$, $q' \geqslant q - (\dim L)m - 2$.
  ($Y'$, $Y'_i$, $z'_1, \ldots, z'_{m+1}$ may be different
for different terms.)
\end{lemma}
\begin{proof}
Let $\alpha_i := \height \varphi(z_i)$.
We will use induction on $\sum_{i=1}^m \alpha_i$.
(The sum will grow.)
 Note that $ \alpha_i \leqslant \theta \leqslant \dim L$.
Denote $I_i := L_{\alpha_i}$, $J_i := L_{\alpha_{i+1}}$.

First, consider the case when $I_1, \ldots, I_m$,
$J_1, \ldots, J_m$ do not satisfy Conditions 1--2.
In this case we can choose $G$-invariant $B$-submodules
$T_i$, $I_i = T_i \oplus J_i$, such that
\begin{equation}\label{EqTqUpperZero}
 [[T_1, \underbrace{L, \ldots, L}_{q_1}], [T_2, \underbrace{L, \ldots, L}_{q_2}], \ldots, [T_m,
 \underbrace{L, \ldots, L}_{q_m}]] = 0
 \end{equation}
 for all $q_i \geqslant 0$.
Rewrite $\varphi(z_i)=a'_i+a''_i$, $a'_i \in T_i$, $a''_i \in J_i$.
Note that $\height a''_i > \height \varphi(z_i)$.
Since $f_{m,q}$ is multilinear, we can rewrite
$\varphi(f_{m,q})$ as a sum of similar
terms $\tilde\varphi(f_{m,q})$ where $\tilde\varphi(z_i)$ equals either
$a'_i$ or $a''_i$. By~(\ref{EqTqUpperZero}), the term where all
$\tilde\varphi(z_i)=a'_i \in T_i$, equals $0$.
For the other terms $\tilde\varphi(f_{m,q})$ we
have $\sum_{i=1}^m \height \tilde\varphi(z_i) > \sum_{i=1}^m \height \varphi(z_i)$.

Thus without lost of generality
we may assume that $I_1, \ldots, I_m$,
$J_1, \ldots, J_m$ satisfy Conditions 1--2.
In this case, $\dim(\Ann(I_1/J_1) \cap \ldots \cap \Ann(I_m/J_m))
\geqslant \dim(L)-d$.
In virtue of Lemma~\ref{LemmaIrrAnnBS},
$$\Ann(I_1/J_1) \cap \ldots \cap \Ann(I_m/J_m)
= B \cap \Ann(I_1/J_1) \cap \ldots \cap \Ann(I_m/J_m) \oplus$$ $$
S \cap \Ann(I_1/J_1) \cap \ldots \cap \Ann(I_m/J_m)\ \oplus\ N.$$
Choose a basis in $B$ that includes a basis of
$B \cap \Ann(I_1/J_1) \cap \ldots \cap \Ann(I_m/J_m)$
and a basis in $S$
that includes the basis of $S \cap \Ann(I_1/J_1) \cap \ldots \cap \Ann(I_m/J_m)$.
Since $f_{m,q}$ is multilinear, we may assume
that only basis elements are substituted for $y_{k\ell}$. Note that $f_{m,q}$
is alternating in $Y_i$. Hence, if $\varphi(f_{m,q})\ne 0$, then for every $1 \leqslant i \leqslant q$
there exists $y_{jk} \in Y_i$ such that
either $$\varphi(y_{jk}) \in B \cap \Ann(I_1/J_1) \cap \ldots \cap \Ann(I_m/J_m)$$
or $$\varphi(y_{jk}) \in S \cap \Ann(I_1/J_1) \cap \ldots \cap \Ann(I_m/J_m).$$

Consider the case when $\varphi(y_{kj}) \in B \cap \Ann(I_1/J_1) \cap \ldots \cap \Ann(I_m/J_m)$
for some $y_{kj}$.
By the corollary from Lemma~\ref{LemmaComplIrrGInv}, we can choose $G$-invariant $B$-submodules $T_k$ such that $I_k = T_k \oplus J_k$. We may assume
that $\varphi(z_k) \in T_k$ since elements of $J_k$ have greater heights.
Therefore $[\varphi(z_k^{g_k}), a] \in T_k \cap J_k$
 for all $a \in B \cap \Ann(I_1/J_1) \cap \ldots \cap \Ann(I_m/J_m)$.
Hence $[\varphi(z_k^{g_k}),a]=0$. Moreover, $B \cap \Ann(I_1/J_1) \cap \ldots \cap \Ann(I_m/J_m)$ is a $G$-invariant ideal of $B$ and $[B,S]=0$. Thus, applying Jacobi's identity
several times, we obtain
$$\varphi([z_k^{g_{k}},y_{k1}^{g_{k1}},\ldots, y_{kj_k}^{g_{kj_k}}]) = 0.$$
Expanding the alternations, we get $\varphi(f_{m,q})=0$.

Consider the case when $\varphi(y_{k\ell}) \in S \cap \Ann(I_1/J_1) \cap \ldots \cap \Ann(I_m/J_m)$
for some $y_{k\ell} \in Y_q$. Expand the alternation $\Alt_q$ in $f_{m,q}$
and rewrite $f_{m,q}$ as a sum of
$$\tilde f_{m,q-1} :=\Alt_{1} \ldots \Alt_{q-1} [[z_1^{g_1}, y_{11}^{g_{11}}, y_{12}^{g_{12}},
  \ldots, y_{1j_1}^{g_{1j_1}}],
 [z_2^{g_2},y_{21}^{g_{21}},y_{22}^{g_{22}},\ldots, y_{2j_2}^{g_{2j_2}}], \ldots,
 $$ $$[ z_m^{g_m}, y_{m1}^{g_{m1}}, y_{m2}^{g_{m2}}, \ldots, y_{mj_m}^{g_{mj_m}}]].$$
 The operator $\Alt_q$ may change indices, however we
 keep the notation $y_{k\ell}$ for the variable with the property
 $\varphi(y_{k\ell}) \in S \cap \Ann(I_1/J_1) \cap \ldots \cap \Ann(I_m/J_m)$.
Now the alternation does not affect $y_{k\ell}$.
Note that $$[z_k^{g_{k}},y_{k1}^{g_{k1}},\ldots, y_{k\ell}^{g_{k\ell}}, \ldots, y_{kj_k}^{g_{kj_k}}] = [z_k^{g_{k}},y_{k\ell}^{g_{k\ell}},y_{k1}^{g_{k1}},\ldots, y_{kj_k}^{g_{kj_k}}] +$$ $$
\sum\limits_{\beta=1}^{\ell-1}
[z_k^{g_{k}},y_{k1}^{g_{k1}},\ldots, y_{k,\beta-1}^{g_{k,\beta-1}}, [y_{k\beta}^{g_{k\beta}}, y_{k\ell}^{g_{k\ell}}],
 y_{k,\beta+1}^{g_{k,\beta+1}},\ldots, y_{k,{\ell-1}}^{g_{k,{\ell-1}}}, y_{k,{\ell+1}}^{g_{k,{\ell+1}}}, \ldots, y_{kj_k}^{g_{kj_k}}].$$

In the first term we replace $[z_k^{g_{k}},y_{k\ell}^{g_{k\ell}}]$
with $z'_k$ and define $\varphi'(z'_k)
:= \varphi([z_k^{g_{k}},y_{k\ell}^{g_{k\ell}}])$, $\varphi'(x) := \varphi(x)$ for other
variables~$x$. Then $\height\varphi'(z'_k) >
\height\varphi(z_k)$ and we can use the inductive assumption.
If $y_{k\beta}\in Y_j$ for some $j$, then we expand the alternation $\Alt_j$
in this term in $\tilde f_{m,q-1}$.
If $\varphi(y_{k\beta}) \in B$, then the term is zero.
If $\varphi(y_{k\beta}) \in S$, then $\varphi([y_{k\beta}^{g_{k\beta}}, y_{k\ell}^{g_{k\ell}}]) \in N$.
We replace $[y_{k\beta}^{g_{k\beta}}, y_{k\ell}^{g_{k\ell}}]$
 with an additional variable $z'_{m+1}$
and define $\psi(z'_{m+1}):=\varphi([y_{k\beta}^{g_{k\beta}},
 y_{k\ell}^{g_{k\ell}}])$, $\psi(x):=\varphi(x)$
 for other variables $x$.
Applying Jacobi's identity several times,
we obtain the polynomial of the desired form. In each inductive step
we reduce $q$ no more than by $1$ and the maximal number of
inductive steps equals $(\dim L)m$. This finishes the proof.
\end{proof}

Since $N$ is a nilpotent ideal,
 $N^{p} = 0$ for some $p\in \mathbb N$.

\begin{lemma}\label{LemmaUpper}
If $\lambda = (\lambda_1, \ldots, \lambda_s) \vdash n$
and $\lambda_{d+1} \geqslant p((\dim L)p+3)$ or $\lambda_{\dim L+1} > 0$, then
$m(L, G, \lambda) = 0$.
\end{lemma}

\begin{proof}
It is sufficient to prove that $e^{*}_{T_\lambda} f \in \Id^G(L)$
for every $f\in V^G_n$ and a Young tableau $T_\lambda$, $\lambda \vdash n$, with
$\lambda_{d+1} \geqslant p((\dim L)p+3)$ or $\lambda_{\dim L+1} > 0$.

Fix some basis of $L$ that is a union of
bases of $B$, $S$, and $N$.
Since polynomials are multilinear, it is
sufficient to substitute only basis elements.
 Note that
$e^{*}_{T_\lambda} = b_{T_\lambda} a_{T_\lambda}$
and $b_{T_\lambda}$ alternates the variables of each column
of $T_\lambda$. Hence if we make a substitution and $
e^{*}_{T_\lambda} f$ does not vanish, then this implies that different basis elements
are substituted for the variables of each column.
But if $\lambda_{\dim L+1} > 0$, then the length of the first column is greater
than $\dim L$. Therefore, $e^{*}_{T_\lambda} f \in \Id^G(L)$.

Consider the case $\lambda_{d+1} \geqslant p((\dim L)p+3)$.
 Let $\varphi$ be a substitution of basis elements for the variables
 $x_1, \ldots, x_n$.
Then $e^{*}_{T_\lambda}f$ can be rewritten as a sum of polynomials $f_{m,q}$
where $1 \leqslant m \leqslant p$, $q \geqslant p((\dim L)p+2)$, and
$z_i$, $2\leqslant i \leqslant m$, are replaced with
 elements of $N$.  (For different terms $f_{m,q}$,
 numbers $m$ and $q$,
 variables $z_i$, $y_{ij}$, and sets $Y_i$ can be different.)
 Indeed, we expand symmetrization on all variables and alternation on
 the variables replaced with elements from  $N$.
    If we have no variables replaced
 with elements from $N$, then we take $m=1$,
 rewrite the polynomial $f$ as a sum of long commutators,
 in each long commutator expand the alternation on the set that
 includes one of the variables in the inner commutator, and denote
 that variable by $z_1$.
 Suppose we have variables replaced
 with elements from $N$. We denote them by $z_k$.
Then, using Jacobi's identity, we can put one of such variables
inside a long commutator and group all the variables,
replaced with elements from  $B \cup S$, around $z_k$
such that each $z_k$ is inside the corresponding long commutator.

 Applying Lemma~\ref{LemmaReduct} many times, we increase $m$.
 The ideal $N$ is nilpotent and $\varphi(f_{p+1,q})=0$
 for every $q$ and a proper homomorphism~$\varphi$.
  Reducing $q$ no more than by $p((\dim L)p+2)$,
 we obtain $\varphi(e^{*}_{T_\lambda}f)=0$.
\end{proof}
Now we can prove
\begin{theorem}\label{TheoremUpper} If $d > 0$, then
there exist constants $C_2 > 0$, $r_2 \in \mathbb R$
such that $c^G_n(L) \leqslant C_2 n^{r_2} d^n$
for all $n \in \mathbb N$. In the case $d=0$, the algebra $L$ is nilpotent.
\end{theorem}
\begin{proof}
Lemma~\ref{LemmaUpper} and~\cite[Lemmas~6.2.4, 6.2.5]{ZaiGia}
imply
$$
\sum_{m(L,G, \lambda)\ne 0} \dim M(\lambda) \leqslant C_8 n^{r_8} d^n
$$
for some constants $C_8, r_8 > 0$.
Together with Theorem~\ref{TheoremMult} this implies the upper bound.
\end{proof}

\section{Alternating polynomials}
\label{SectionAlt}

In this section we prove auxiliary propositions needed to obtain the lower bound.

\begin{lemma}\label{LemmaTwoColumns}
Let $\alpha_1, \alpha_2, \ldots, \alpha_q$, $\beta_1, \ldots, \beta_q \in F$, $
1 \leqslant k \leqslant q$,
 $\alpha_i\ne 0$ for $1 \leqslant i < k$,
$\alpha_k=0$, and $\beta_k\ne 0$. Then there exists such
$\gamma \in F$ that $\alpha_i + \gamma \beta_i \ne 0$
for all $1 \leqslant i \leqslant k$.
\end{lemma}
\begin{proof} It is sufficient to choose
$\gamma \notin \left\lbrace -\frac{\alpha_1}{\beta_1},
\ldots, -\frac{\alpha_{k-1}}{\beta_{k-1}}, 0\right\rbrace$. It is possible to do
since $F$ is infinite.
\end{proof}

Let $F \langle X | G \rangle$ be the
free associative algebra over $F$ with free formal generators $x^g_j$, $j\in\mathbb N$,
 $g \in G$. Define $(x^g_j)^h = x^{hg}_j$
 for $h \in G$.
 Then $F \langle X | G \rangle$ becomes the \textit{free associative $G$-algebra} with
 free generators $x_j = x_j^1$, $j \in \mathbb N$, $1 \in G$.
Denote by $P^G_n$, $n\in \mathbb N$, the subspace of associative multilinear $G$-polynomials
in variables $x_1, \ldots, x_n$.
In other words, $$P^G_n = \left\lbrace\sum_{\sigma \in S_n,\,
 g_1, \ldots, g_n \in G} \alpha_{\sigma, g_1, \ldots, g_n}\, x_{\sigma(1)}^{g_1} x_{\sigma(2)}^{g_2}
\ldots x_{\sigma(n)}^{g_n} \, \biggl| \,
\alpha_{\sigma, g_1, \ldots, g_n} \in F \right\rbrace.$$

\begin{lemma}\label{LemmaS}
Let $L_0=B_0 \oplus R_0$ be a reductive Lie algebra
with $G$-action, $B_0$ be a maximal semisimple $G$-subalgebra, and
$R_0$ be the center of $L_0$ with a basis $r_1$, $r_2$, \ldots, $r_t$.
Let $M$ be a faithful finite dimensional irreducible $L_0$-module with $G$-action.
Denote the corresponding representation $L_0 \to \mathfrak{gl}(M)$ by $\varphi$.
Then there exists such alternating in $x_1, x_2, \ldots, x_t$
 polynomial $f \in P^G_t$  that
$f(\varphi(r_1), \ldots, \varphi(r_t))$ is a nondegenerate operator on~$M$.
\end{lemma}
\begin{proof} By Lemma~\ref{LemmaRedIrr},
$M=M_1\oplus\ldots \oplus M_q$ where $M_j$ are $L_0$-submodules
and $r_i$ acts on each $M_j$ as a scalar operator.
Note that it is sufficient to prove that for each $j$ there exists
such alternating in $x_1, x_2, \ldots, x_t$
 polynomial $f_j \in P^G_t$  that
$f_j(\varphi(r_1), \ldots, \varphi(r_t))$ multiplies
each element of~$M_j$ by a nonzero scalar.
Indeed, in this case Lemma~\ref{LemmaTwoColumns} implies the existence
of such $f=\gamma_1 f_1 +\ldots + \gamma_q f_q$, $\gamma_i \in F$,
that $f(\varphi(r_1), \ldots, \varphi(r_t))$ acts
 on each $M_i$ as a nonzero scalar.

 Denote by $p_i \in \End_F(M)$ the projection
  on $M_i$ along $\bigoplus_{k\ne i} M_k$.
 Fix $1 \leqslant j \leqslant q$. By Lemma~\ref{LemmaRedIrr}, Proposition~\ref{RedGAction}, we can choose such $g_i \in G$
 that $M_i^{g_i}=M_j$, $1 \leqslant i \leqslant q$. Then $p_i^{g_i}=p_j$.
 Consider $\tilde f_j := \sum_{\sigma \in S_q} (\sign \sigma)
 x_{\sigma(1)}^{g_1} x_{\sigma(2)}^{g_2}
 \ldots x_{\sigma(q)}^{g_q}$.
 Note that either $ p_{\sigma(1)}^{g_1} p_{\sigma(2)}^{g_2}
 \ldots p_{\sigma(q)}^{g_q} = 0$
 or $ p_{\sigma(1)}^{g_1} p_{\sigma(2)}^{g_2}
 \ldots p_{\sigma(q)}^{g_q} = p_k$ for some $1 \leqslant k \leqslant s$.
 Now we prove that $ p_{\sigma(1)}^{g_1} p_{\sigma(2)}^{g_2}
 \ldots p_{\sigma(q)}^{g_q} = p_j$ if and only if
 $\sigma(i)=i$ for all $1 \leqslant i \leqslant q$.
 Indeed, $p_{\sigma(i)}^{g_i} = p_j$
 if and only if $M_{\sigma(i)}^{g_i}=M_j$. Hence
 $\sigma(i)=i$. This implies that $\tilde f_j
 (p_1, \ldots, p_q)$ acts as an identical map
 on $M_j$.

 We can choose $i_{t+1}, \ldots, i_q$
 such that $\varphi(r_1), \varphi(r_2), \ldots,
 \varphi(r_t)$, $p_{i_{t+1}},
 \ldots, p_{i_q}$ form a basis in
 $ \langle p_1, \ldots, p_q \rangle_F$.
 Then $\tilde f_j
 (\varphi(r_1), \varphi(r_2), \ldots,
 \varphi(r_t), p_{i_{t+1}},
 \ldots, p_{i_q})$
 acts as a nonzero scalar on $M_j$.
 If $t=q$, then we define $f_j = \tilde f_j$.
 Suppose $t < q$.
 Since the projections commute, we can rewrite $$
 \tilde f_j(\varphi(r_1), \varphi(r_2), \ldots,
 \varphi(r_t), p_{i_{t+1}},
 \ldots, p_{i_q}) = \sum_{i=1}^q \hat f_i(\varphi(r_1), \varphi(r_2), \ldots,
 \varphi(r_t))p_i$$
 where $\hat f_i \in P^G_t$ are
 alternating in $x_1, x_2, \ldots,
 x_t$. Hence $\hat f_j(\varphi(r_1), \varphi(r_2), \ldots,
 \varphi(r_t))$ acts on $M_j$ as a nonzero
 scalar operator.
 We define $f_j:=\hat f_j$.
\end{proof}

 Let $L_0$ be a Lie algebra with $G$-action, $M$ be $L_0$-module with $G$-action, $\varphi \colon  L_0 \to \mathfrak{gl}(M)$ be the corresponding representation. A polynomial $f(x_1, \ldots, x_n)\in F\langle X | G \rangle$ is a \textit{$G$-identity} of $\varphi$ if $f(\varphi(a_1), \ldots, \varphi(a_n))=0$ for all $a_i \in L_0$. The set $\Id^{G}(\varphi)$ of all $G$-identities of $\varphi$ is a two-sided ideal in $F\langle X | G \rangle$ invariant under $G$-action.

Lemma~\ref{LemmaAlternateFirst} is an analog of~\cite[Lemma~1]{GiaSheZai}.

\begin{lemma}\label{LemmaAlternateFirst}
Let $L_0$ be a Lie algebra with $G$-action,
$M$ be a faithful finite dimensional irreducible $L_0$-module with $G$-action,
and $\varphi \colon L_0 \to \mathfrak{gl}(M)$ be the corresponding
representation. Then for some $n\in\mathbb N$ there exists
 a polynomial $f \in P^G_n \backslash \Id^G(\varphi)$
  alternating in $\lbrace x_1, \ldots, x_\ell \rbrace$ and in $\lbrace y_1, \ldots, y_\ell \rbrace \subseteq \lbrace x_{\ell+1}, \ldots, x_n \rbrace$ where $\ell=\dim L_0$.
\end{lemma}
\begin{proof}
Since $M$ is irreducible, by the density theorem,
 $\End_F(M) \cong M_q(F)$ is generated by operators from $G$
and $\varphi(L_0)$. Here $q := \dim M$. Consider Regev's polynomial
$$\hat f(x_1, \ldots, x_{q^2}; y_1, \ldots, y_{q^2})
:=\sum_{\substack{\sigma \in S_q, \\ \tau \in S_q}} (\sign(\sigma\tau))
x_{\sigma(1)}\ y_{\tau(1)}\ x_{\sigma(2)}x_{\sigma(3)}x_{\sigma(4)}
\ y_{\tau(2)}y_{\tau(3)}y_{\tau(4)}\ldots
$$
$$ x_{\sigma\left(q^2-2q+2\right)}\ldots x_{\sigma\left(q^2\right)}
\ y_{\tau\left(q^2-2q+2\right)}\ldots y_{\tau\left(q^2\right)}.
$$ This is a central polynomial~\cite[Theorem~5.7.4]{ZaiGia}
for $M_k(F)$, i.e. $\hat f$ is not a polynomial identity for $M_q(F)$
and its values belong to the center of $M_q(F)$.

Let $a_1, \ldots, a_\ell$ be a basis of $L_0$.
Denote by $\rho$ the representation $G \to \mathrm{GL}(M)$.
Note that if we have the product of elements of
$\varphi(L_0)$ and $\rho(G)$,
we can always move the elements from $\rho(G)$
to the right, using $\rho(g)a=a^g \rho(g)$
for $g\in G$ and $a \in \varphi(L_0)$.
Then
$\varphi(a_1), \ldots, \varphi(a_\ell)$, $\left(\varphi\left(a_{i_{11}}\right)
\ldots \varphi\left(a_{i_{1,m_1}}\right)\right)\rho(g_1)$,
\ldots, $\left(\varphi\left(a_{i_{r,1}}\right)
\ldots \varphi\left(a_{i_{r,m_r}}
\right)\right)\rho(g_r)$, is a basis of $\End_F(M)$ for appropriate $i_{jk} \in \lbrace 1,2, \ldots, \ell \rbrace$, $g_j \in G$, since $\End_F(M)$ is generated by operators from $G$
and $\varphi(L_0)$. We replace $x_{\ell+j}$ with $z_{j1}
z_{j2} \ldots z_{j,m_j} \rho(g_j)$ and $y_{\ell+j}$ with $z'_{j1}
z'_{j2} \ldots z'_{j,m_j} \rho(g_j)$ in $\hat f$ and denote the expression
obtained by $\tilde f$. Using $\rho(g)a=a^g\rho(g)$ again,
we can move all $\rho(g)$, $g \in G$, in $\tilde f_q$ to the right and rewrite
 $\tilde f$ as $\sum_{g\in G}{f_{g}\, \rho(g)}$
where each $f_{g} \in P^G_{2\ell+2\sum_{j=1}^r m_j}$ is
an alternating in $x_1, \ldots, x_\ell$ and
in $y_1, \ldots, y_\ell$ polynomial.
Note that $\tilde f$ becomes a nonzero scalar operator on~$M$ under
the substitution $x_i=y_i=\varphi(a_i)$ for $ 1 \leqslant i \leqslant \ell$
and $z_{jk}=z_{jk}'=\varphi(a_{i_{jk}})$ for $1 \leqslant j \leqslant r$, $1 \leqslant k \leqslant m_j$.
 Thus $f_{g} \notin \Id^G(\varphi)$
for some $g \in G$ and we can take $f=f_{g}$.
\end{proof}

Let $k\ell \leqslant n$ where $k,\ell,n \in \mathbb N$ are some numbers.
 Denote by $Q^G_{\ell,k,n} \subseteq P^G_n$
the subspace spanned by all polynomials that are alternating in
$k$ disjoint subsets of variables $\{x^i_1, \ldots, x^i_\ell \}
\subseteq \lbrace x_1, x_2, \ldots, x_n\rbrace$, $1 \leqslant i \leqslant k$.

Theorem~\ref{TheoremAlternateFinal}
 is an analog of~\cite[Theorem~1]{GiaSheZai}.

\begin{theorem}\label{TheoremAlternateFinal}
Let $L_0=B_0 \oplus R_0$ be a reductive Lie algebra
with $G$-action over an algebraically closed field $F$ of
 characteristic $0$, $B_0$ be a maximal semisimple $G$-subalgebra,
$R_0$ be the center of $L_0$, and $\dim L_0 = \ell$.
Let $M$ be a faithful finite dimensional irreducible $L_0$-module with $G$-action.
Denote the corresponding representation $L_0 \to \mathfrak{gl}(M)$ by $\varphi$.
Then there exists $T \in \mathbb Z_+$ such that
for any $k \in \mathbb N$
there exists $f \in Q^G_{\ell, 2k, 2k\ell+T} \backslash \Id^G(\varphi)$.
\end{theorem}
\begin{proof}
Let $f_1=f_1(x_1,\ldots, x_\ell,\ y_1,\ldots, y_\ell,
z_1, \ldots, z_T)$ be the polynomial from Lemma~\ref{LemmaAlternateFirst}
alternating in $x_1,\ldots, x_\ell$ and in $y_1,\ldots, y_\ell$.
Since $f_1 \in Q^G_{\ell, 2, 2\ell+T} \backslash \Id^G(\varphi)$,
 we may assume that $k > 1$. Note that
$$
f^{(1)}_1(u_1, v_1, x_1, \ldots, x_\ell,\ y_1,\ldots, y_\ell,
z_1, \ldots, z_T) :=$$ $$
\sum^\ell_{i=1} f_1(x_1, \ldots, [u_1, [v_1, x_i]],  \ldots, x_\ell,\ y_1,\ldots, y_\ell,
z_1, \ldots, z_T)$$
is alternating in $x_1,\ldots, x_\ell$ and in $y_1,\ldots, y_\ell$
and $$
f^{(1)}_1(\bar u_1, \bar v_1, \bar x_1, \ldots, \bar x_\ell,\
\bar y_1,\ldots, \bar y_\ell,
\bar z_1, \ldots, \bar z_T) =$$
$$
 \tr(\ad_{\varphi(L_0)} \bar u_1 \ad_{\varphi(L_0)} \bar v_1)
f_1(\bar x_1, \bar x_2, \ldots, \bar x_\ell,\ \bar y_1,\ldots, \bar y_\ell,
\bar z_1, \ldots, \bar z_T)
$$
 for any substitution of elements from $\varphi(L_0)$
 since we may assume $\bar x_1, \ldots, \bar x_\ell$ to be different basis elements.
Here $(\ad a) b = [a,b]$.

Let $$
f^{(j)}_1(u_1, \ldots, u_j, v_1, \ldots, v_j, x_1, \ldots, x_\ell,\ y_1,\ldots, y_\ell,
z_1, \ldots, z_T) :=$$ $$
\sum^\ell_{i=1} f^{(j-1)}_1(u_1, \ldots,  u_{j-1}, v_1, \ldots, v_{j-1},
 x_1, \ldots, [u_j, [v_j, x_i]],  \ldots, x_\ell,\ y_1,\ldots, y_\ell,
z_1, \ldots, z_T),$$
$2 \leqslant j \leqslant s$, $s = \dim B$. Note that
if we substitute an element from $\varphi(R_0)$ for $u_i$ or $v_i$,
then $f^{(j)}_1$ vanish since $R_0$ is the center of $L_0$.
Again,
$$
f^{(j)}_1(\bar u_1, \ldots, \bar u_j, \bar v_1, \ldots, \bar v_j, \bar x_1, \ldots, \bar x_\ell,\ \bar y_1,\ldots, \bar y_\ell, \bar z_1, \ldots, \bar z_T) =$$
$$ \tr(\ad_{\varphi(L_0)} \bar u_1 \ad_{\varphi(L_0)} \bar v_1)
 \tr(\ad_{\varphi(L_0)} \bar u_2 \ad_{\varphi(L_0)} \bar v_2)
 \ldots
 \tr(\ad_{\varphi(L_0)} \bar u_j \ad_{\varphi(L_0)} \bar v_j)\cdot
$$
\begin{equation}\label{EqKilling}
\cdot
f_1(\bar x_1, \bar x_2, \ldots, \bar x_\ell,\ \bar y_1,\ldots, \bar y_\ell,
\bar z_1, \ldots, \bar z_T)
\end{equation}

Let $h$ be the polynomial from Lemma~\ref{LemmaS}.
We define
$$ f_2(u_1, \ldots, u_\ell, v_1, \ldots, v_\ell,
x_1, \ldots, x_\ell, y_1, \ldots, y_\ell, z_1, \ldots, z_T) :=
$$ $$\sum_{\sigma, \tau \in S_\ell}
\sign(\sigma\tau)
f^{(s)}_1(u_{\sigma(1)}, \ldots, u_{\sigma(s)}, v_{\tau(1)}, \ldots, v_{\tau(s)}, x_1, \ldots, x_\ell,\ y_1,\ldots, y_\ell,
z_1, \ldots, z_T)$$ $$\cdot h(u_{\sigma(s+1)}, \ldots, u_{\sigma(\ell)})
h(v_{\tau(s+1)}, \ldots, v_{\tau(\ell)}).$$
Then $f_2 \in Q^G_{\ell, 4, 4\ell+T}$. Suppose $a_1, \ldots, a_s \in \varphi(B_0)$
and $a_{s+1}, \ldots, a_\ell \in \varphi(R_0)$ form a basis of $\varphi(L_0)$.
Consider a substitution $x_i=y_i=u_i=v_i=a_i$, $1 \leqslant i \leqslant \ell$.
Suppose that the values $z_j=\bar z_j$, $1 \leqslant j \leqslant T$, are chosen
in such a way that $f_1(a_1, \ldots, a_\ell, a_1, \ldots, a_\ell,
\bar z_1, \ldots, \bar z_T)\ne 0$. We claim that $f_2$ does not vanish either.
Indeed,
$$ f_2(a_1, \ldots, a_\ell, a_1, \ldots, a_\ell,
a_1, \ldots, a_\ell, a_1, \ldots, a_\ell, \bar z_1, \ldots, \bar z_T) = $$
$$\sum_{\sigma, \tau \in S_\ell}
\sign(\sigma\tau)
f^{(s)}_1(a_{\sigma(1)}, \ldots, a_{\sigma(s)}, a_{\tau(1)}, \ldots, a_{\tau(s)}, a_1, \ldots, a_\ell,\ a_1,\ldots, a_\ell,
\bar z_1, \ldots, \bar z_T)$$ $$\cdot h(a_{\sigma(s+1)}, \ldots, a_{\sigma(\ell)})
h(a_{\tau(s+1)}, \ldots, a_{\tau(\ell)})=$$
$$\left(\sum_{\sigma, \tau \in S_s}
\sign(\sigma\tau)
f^{(s)}_1(a_{\sigma(1)}, \ldots, a_{\sigma(s)}, a_{\tau(1)}, \ldots, a_{\tau(s)}, a_1, \ldots, a_\ell,\ a_1,\ldots, a_\ell,
\bar z_1, \ldots, \bar z_T)\right)\cdot$$
 $$ \left(
\sum_{ \pi, \omega \in S\lbrace s+1, \ldots,
\ell \rbrace} \sign(\pi\omega)
 h(a_{\pi(s+1)}, \ldots, a_{\pi(\ell)})
h(a_{\omega(s+1)}, \ldots, a_{\omega(\ell)})\right)$$
 since $a_j$, $s < j \leqslant \ell$,
belong to the center of $\varphi(L_0)$
and $f^{(s)}_j$ vanishes if we substitute
such $a_i$ for $u_i$ or $v_i$.
Here $S\lbrace s+1, \ldots,
\ell \rbrace$ is the symmetric group on $\lbrace s+1, \ldots,
\ell \rbrace$. Note that $h$ is alternating. Using~(\ref{EqKilling}), we obtain
$$ f_2(a_1, \ldots, a_\ell, a_1, \ldots, a_\ell,
a_1, \ldots, a_\ell, a_1, \ldots, a_\ell, \bar z_1, \ldots, \bar z_T) = $$
$$ \left(
\sum_{\sigma, \tau \in S_s}
\sign(\sigma\tau) \tr(\ad_{\varphi(L_0)} a_{\sigma(1)}
\ad_{\varphi(L_0)} a_{\tau(1)})  \ldots \tr(\ad_{\varphi(L_0)} a_{\sigma(s)}
\ad_{\varphi(L_0)} a_{\tau(s)}) \right)\cdot$$ $$
f_1(a_1, \ldots, a_\ell,\ a_1,\ldots, a_\ell,
\bar z_1, \ldots, \bar z_T)  ((\ell-s)!)^2
\left(h(a_{s+1}, \ldots, a_\ell)\right)^2.
$$
 Note that
$$\sum_{\sigma, \tau \in S_s}
\sign(\sigma\tau) \tr(\ad_{\varphi(L_0)} a_{\sigma(1)}
\ad_{\varphi(L_0)} a_{\tau(1)})  \ldots \tr(\ad_{\varphi(L_0)} a_{\sigma(s)}
\ad_{\varphi(L_0)} a_{\tau(s)})
=$$ $$\sum_{\sigma, \tau \in S_s}
\sign(\sigma\tau) \tr(\ad_{\varphi(L_0)} a_{1}
\ad_{\varphi(L_0)} a_{\tau\sigma^{-1}(1)})  \ldots
 \tr(\ad_{\varphi(L_0)} a_{s}
\ad_{\varphi(L_0)} a_{\tau\sigma^{-1}(s)})\mathrel{\stackrel{(\tau'=\tau\sigma^{-1})}{=}}$$
$$\sum_{\sigma, \tau' \in S_s}
\sign(\tau') \tr(\ad_{\varphi(L_0)} a_{1}
\ad_{\varphi(L_0)} a_{\tau'(1)})  \ldots
 \tr(\ad_{\varphi(L_0)} a_{s}
\ad_{\varphi(L_0)} a_{\tau'(s)})=$$
$$s!\det(\tr(\ad_{\varphi(L_0)} a_i \ad_{\varphi(L_0)} a_j))_{i,j=1}^s=
s!\det(\tr(\ad_{\varphi(B_0)} a_i \ad_{\varphi(B_0)} a_j))_{i,j=1}^s \ne 0$$
since the Killing form $\tr(\ad x \ad y)$ of the semisimple
Lie algebra $\varphi(B_0)$ is nondegenerate.
Thus $$ f_2(a_1, \ldots, a_\ell, a_1, \ldots, a_\ell,
a_1, \ldots, a_\ell, a_1, \ldots, a_\ell, \bar z_1, \ldots, \bar z_T) \ne 0. $$
Note that if $f_1$ is alternating in some of $z_1,\ldots, z_T$,
the polynomial $f_2$
is alternating in those variables too.
Thus if we apply the same procedure to
$f_2$ instead of $f_1$, we obtain $f_3 \in Q^G_{\ell, 6, 6\ell+T}$.
Analogously, we define $f_4$ using $f_3$, $f_5$ using $f_4$, etc.
Eventually, we obtain
$f=f_k \in Q^G_{\ell, 2k, 2k\ell+T} \backslash \Id^G(\varphi)$.
\end{proof}

\section{Lower bound}
\label{SectionLower}

By the definition of $d=d(L)$,
 there exist $G$-invariant ideals $I_1$, $I_2$, \ldots, $I_r$,
$J_1$, $J_2$, \ldots, $J_r$, $r \in \mathbb Z_+$, of the algebra $L$,
satisfying Conditions 1--2, $J_k \subseteq I_k$, such that
$$d = \dim \frac{L}{\Ann(I_1/J_1) \cap \dots \cap \Ann(I_r/J_r)}.$$
We consider the case $d > 0$.

Without loss of generality we may assume that
$$ \bigcap\limits_{k=1}^r \Ann(I_k/J_k) \ne
\bigcap\limits_{\substack{\phantom{,}k=1,\\ k\ne\ell}}^r \Ann(I_k/J_k)$$
for all $1 \leqslant \ell \leqslant r$.
In particular, $L$ has nonzero action on each $I_k/J_k$.

Our aim is to present a partition $\lambda \vdash n$
with $m(L, G, \lambda)\ne 0$ such that $\dim M(\lambda)$
has the desired asymptotic behavior.
We will glue alternating polynomials constructed
 in Theorem~\ref{TheoremAlternateFinal}
for faithful irreducible modules
over reductive algebras. In order to do this,
we have to choose the reductive algebras.

\begin{lemma}\label{LemmaChooseReduct}
There exist $G$-invariant ideals $B_1, \ldots, B_r$
in $B$ and $G$-invariant subspaces
 $\tilde R_1,\ldots,\tilde R_r \subseteq S$
(some of $\tilde R_i$ and $B_j$ may be zero)
such that
\begin{enumerate}
\item $B_1+ \ldots + B_r=B_1\oplus \ldots \oplus B_r$;
\item $\tilde R_1+ \ldots + \tilde R_r=\tilde R_1\oplus \ldots \oplus \tilde R_r$;
\item $\sum\limits_{k=1}^r \dim (B_k\oplus  \tilde R_k) = d$;
\item $I_k/J_k$ is a faithful
$(B_k\oplus\tilde R_k\oplus N)/N$-module;
\item $I_k/J_k$ is an irreducible
$\left(\sum_{i=1}^r (B_i\oplus \tilde R_i)\oplus N
\right)/N$-module with $G$-action;
\item $B_i I_k/J_k = \tilde R_i I_k/J_k = 0$ for $i > k$.
\end{enumerate}
\end{lemma}
\begin{proof}
Consider $N_\ell := \bigcap\limits_{k=1}^\ell \Ann (I_k/J_k)$,
$1 \leqslant \ell \leqslant r$, $N_0 = L$.
Note that $N_{\ell}$ are $G$-invariant.
Since $B$ is semisimple, we can choose such
$G$-invariant ideals $B_\ell$
that $N_{\ell-1} \cap B =
 B_\ell \oplus (N_\ell \cap B)$.
Also we can choose such $G$-invariant subspaces
 $\tilde R_\ell$ that
 $N_{\ell-1} \cap S = \tilde R_\ell
  \oplus (N_\ell \cap S)$.
Hence Properties 1, 2, 6 hold.

 By Lemma~\ref{LemmaIrrAnnBS},
 $N_k = (N_k \cap B) \oplus  (N_k \cap S) \oplus  N$.
 Thus Property~4 holds. Furthermore,
 $$N_{\ell-1} = B_\ell \oplus  (N_\ell \cap B) \oplus
  \tilde R_\ell \oplus
   (N_\ell \cap S) \oplus  N
   = (B_\ell\oplus \tilde R_\ell)\oplus N_\ell$$
   (direct sum of subspaces).
   Hence $L = \left(\bigoplus_{i=1}^r (B_i\oplus \tilde R_i)\right)
   \oplus  N_r$,  and
   Properties 3 and 5 hold too.
\end{proof}

Let $A$ be the associative subalgebra
in $\End_F (L)$ generated by operators from $\ad L$
and $G$. Then $J(A)^p=0$ for some $p\in\mathbb N$.
Denote by $A_2$ a subalgebra of $\End_F(L)$ generated by $\ad L$ only.
 Let $a_{\ell 1}, \ldots, a_{\ell, k_\ell}$ be a basis of $\tilde R_\ell$.

\begin{lemma}\label{LemmaJordanDecomp}
There exist decompositions $\ad a_{ij} = c_{ij} + d_{ij}$,
$1 \leqslant i \leqslant r$, $1 \leqslant j \leqslant k_i$,
such that $c_{ij} \in A$ acts as a diagonalizable operator on $L$, $d_{ij} \in J(A)$,
 elements $c_{ij}$ commute with each other,
 and $c_{ij}$ and $d_{ij}$ are polynomials in $\ad a_{ij}$.
 Moreover, $R_\ell := \langle c_{\ell1}, \ldots, c_{\ell, k_\ell} \rangle_F$
 are $G$-invariant subspaces in $A$.

\end{lemma}
\begin{proof}
Consider the solvable $G$-invariant Lie algebra $(\ad R)+J(A)$.
 In virtue of the Lie theorem, there exists
 a basis in $L$ in which all the operators from $(\ad R)+J(A)$ have
 upper triangular matrices. Denote the corresponding
 embedding $A  \hookrightarrow M_m(F)$ by $\psi$.
 Here $m := \dim L$.

Let $A_1$ be the associative algebra generated by $\ad a_{ij}$,
$1 \leqslant i \leqslant r$, $1 \leqslant j \leqslant k_i$.
This algebra is $G$-invariant since for every fixed $i$
the elements $a_{ij}$, $1 \leqslant j \leqslant k_i$,
form a basis of the $G$-invariant subspace $\tilde R_i$.
 By the $G$-invariant Wedderburn~--- Malcev theorem~\cite[Theorem~1, Remark~1]{Taft},  $A_1 =
 \tilde A_1 \oplus J(A_1)$
 (direct sum of subspaces)
where $\tilde A_1$ is a $G$-invariant semisimple subalgebra
of $A_1$.
Since $\psi(\ad R) \subseteq \mathfrak{t}_m(F)$, we have
$\psi(A_1) \subseteq UT_m(F)$. Here $UT_m(F)$ is the associative algebra
 of upper triangular matrices $m\times m$.
 There is a decomposition $$UT_m(F) = Fe_{11}\oplus Fe_{22}\oplus
 \dots\oplus Fe_{mm}\oplus \tilde N$$
 where $$\tilde N := \langle e_{ij} \mid 1 \leqslant i < j \leqslant m \rangle_F$$
 is a nilpotent ideal. Thus there is no subalgebras in $A_1$
 isomorphic to $M_2(F)$ and
  $\tilde A_1=Fe_1 \oplus \dots \oplus Fe_t$
    for some idempotents $e_i \in A_1$.
Denote for every $a_{ij}$  its component in $J(A_1)$ by $d_{ij}$
and its component in $Fe_1 \oplus \dots \oplus Fe_t$ by $c_{ij}$.
Note that $e_i$ are commuting diagonalizable operators. Thus they
have a common basis of eigenvectors in $L$ and $c_{ij}$
are commuting diagonalizable operators too.
Moreover $$\ad a^g_{ij} = c^g_{ij}+d^g_{ij} \in \langle \ad a_{i\ell}
\mid 1 \leqslant \ell \leqslant k_i \rangle_F
\subseteq \langle c_{i\ell}
\mid 1 \leqslant \ell \leqslant k_i \rangle_F
\oplus \langle d_{i\ell}
\mid 1 \leqslant \ell \leqslant k_i \rangle_F$$
for all $g \in G$. Thus $R_i$ is $G$-invariant.

We claim that the space $J(A_1)+J(A)$ generates a nilpotent
$G$-invariant ideal $I$ in $A$.
First, $\psi(J(A_1)), \psi(J(A)) \subseteq UT_m(F)$
and consist of nilpotent elements. Thus the corresponding
matrices have zero diagonal elements and
$\psi(J(A_1)), \psi(J(A)) \subseteq \tilde N$.
Denote $\tilde N_k := \langle e_{ij} \mid i+k \leqslant j \rangle_F \subseteq \tilde N$.
Then $$\tilde N = \tilde N_1 \supsetneqq \tilde N_2 \supsetneqq \ldots \supsetneqq \tilde N_{m-1} \supsetneqq \tilde N_m = \lbrace 0\rbrace.$$
Let $\height_{\tilde N} a := k$ if $\psi(a) \in \tilde N_k$, $\psi(a) \notin \tilde N_{k+1}$.

Recall that $(J(A))^p =0$.
We claim that $I^{m+p} = 0$.
Let $\rho \colon G \to \mathrm{GL}(L)$ be the $G$-action on $L$.
Using the property
\begin{equation}\label{EqGMoveRight}
\rho(g)a=a^g\rho(g)
\end{equation}
where $a\in A_2$, $g\in G$, we obtain that
 the space $I^{m+p}$
 is a span of $h_1 j_1 h_2 j_2 \ldots j_{m+p} h_{m+p+1} \rho(g)$
 where $j_k \in J(A_1) \cup J(A)$, $h_k \in A_2 \cup \lbrace 1\rbrace$,
 $g \in G$.
 If at least $p$ elements $j_k$ belong to $J(A)$,
 then the product equals $0$.
Thus we may assume that at least $m$
elements $j_k$ belong to $J(A_1)$.

 Let $j_i \in J(A_1)$, $h_i \in A_2 \cup \lbrace 1\rbrace$.
We prove by induction on $\ell$ that
$j_1 h_1 j_2 h_2 \ldots h_{\ell-1} j_{\ell}$
can be expressed as a sum of $\tilde j_1 \tilde j_2 \ldots \tilde j_\alpha j'_1 j'_2\ldots j'_\beta
a$
where $\tilde j_i \in J(A_1)$, $j'_i \in J(A)$,
$a \in A_2 \cup \lbrace 1\rbrace$,
 and $\alpha+\sum_{i=1}^\beta \height_{\tilde N} j'_i \geqslant \ell$.
 Indeed, suppose that
  $j_1 h_1 j_2 h_2 \ldots h_{\ell-2} j_{\ell-1}$
 can be expressed as a sum of $\tilde j_1 \tilde j_2 \ldots \tilde j_\gamma j'_1 j'_2\ldots j'_\varkappa
a$
where $\tilde j_i \in J(A_1)$, $j'_i \in J(A)$,
$a \in A_2 \cup \lbrace 1\rbrace$,
 and $\gamma+\sum_{i=1}^\varkappa \height_{\tilde N} j'_i \geqslant \ell-1$.
 Then
 $j_1 h_1 j_2 h_2 \ldots j_{\ell-1} h_{\ell-1}j_{\ell}$
 is a sum of
 $$\tilde j_1 \tilde j_2 \ldots \tilde j_\gamma j'_1 j'_2\ldots j'_\varkappa
a h_{\ell-1}j_{\ell} =
\tilde j_1 \tilde j_2 \ldots \tilde j_\gamma j'_1 j'_2\ldots j'_\varkappa
[ah_{\ell-1}, j_{\ell}] + \tilde j_1 \tilde j_2 \ldots \tilde j_\gamma j'_1 j'_2\ldots j'_\varkappa
j_{\ell} (a h_{\ell-1}).$$
Note that, in virtue of the Jacobi identity
and Lemma~\ref{LemmaLR},  $[ah_{\ell-1}, j_{\ell}] \in
J(A)$. Thus it is sufficient to consider only the second term.
However
$$\tilde j_1 \tilde j_2 \ldots \tilde j_\gamma j'_1 j'_2\ldots j'_\varkappa
j_{\ell} (a h_{\ell-1})
= \tilde j_1 \tilde j_2 \ldots \tilde j_\gamma j_{\ell} j'_1 j'_2\ldots j'_\varkappa
 (a h_{\ell-1})+$$ $$\sum_{i=1}^{\varkappa}
\tilde j_1 \tilde j_2 \ldots \tilde j_\gamma  j'_1 j'_2\ldots j'_{i-1}[j'_{i}, j_\ell]
j'_{i+1}\ldots j'_\varkappa (a h_{\ell-1}).$$
Since $[j'_{i}, j_\ell] \in J(A)$ and $\height_{\tilde N} [j'_{i}, j_\ell] \geqslant 1+ \height_{\tilde N} j'_i$,
all the terms have  the desired form.
 Therefore, $$j_1 h_1 j_2 h_2 \ldots j_{m-1} h_{m-1}j_{m}
 \in \psi^{-1}(\tilde N_m) = \lbrace 0 \rbrace,$$ $I^{m+p}=0$, and $$
 J(A) \subseteq J(A_1)+J(A) \subseteq I \subseteq J(A).$$ In particular,
$d_{ij} \in J(A_1) \subseteq J(A)$.
\end{proof}

Denote $$\tilde B := \left(\bigoplus_{i=1}^r \ad B_i\right)\oplus \langle c_{ij} \mid 1\leqslant i \leqslant r,  1 \leqslant j \leqslant k_i
  \rangle_F,$$
  $$\tilde B_0 := (\ad B)\oplus \langle c_{ij} \mid 1\leqslant i \leqslant r,  1 \leqslant j \leqslant k_i
  \rangle_F \subseteq A.$$

\begin{lemma}\label{LemmaBcReducible}
The space $L$ is a completely reducible $\tilde B_0$-module with $G$-action.
Moreover, $L$ is a completely reducible $(\ad B_k)\oplus R_k$-module
with $G$-action for any $1 \leqslant k \leqslant r$.
\end{lemma}
\begin{proof} By Lemma~\ref{LemmaComplIrrGInv}, it is sufficient to show that
$L$ is a completely reducible $\tilde B_0$-module
and a completely reducible $(\ad B_k)\oplus R_k$-module
disregarding the $G$-action.
 The elements $c_{ij}$ are diagonalizable on $L$ and commute.
  Therefore, an eigenspace of any $c_{ij}$ is invariant
   under the action of other $c_{k\ell}$. Using induction,
    we split $L = \bigoplus_{i=1}^\alpha W_i$
where $W_i$ are intersections of eigenspaces of $c_{k\ell}$
and elements $c_{k\ell}$ act as scalar operators on $W_i$.
In virtue of Lemmas~\ref{LemmaRS}, \ref{LemmaJordanDecomp}, and the Jacobi identity,
 $[c_{ij}, \ad B]=0$. Thus $W_i$ are $B$-submodules
 and $L$ is a completely reducible
 $\tilde B_0$-module and $(\ad B_k)\oplus R_k$-module since $B$ and $B_k$ are semisimple.
\end{proof}

\begin{lemma}\label{LemmaSiProperties}
There exist complementary subspaces $I_k=
\tilde T_k \oplus J_k$ such that
\begin{enumerate}
\item $\tilde T_k$ is a $B$-submodule and an irreducible $\tilde B$-submodule with $G$-action;
\item $\tilde T_k$ is a completely reducible faithful $(\ad B_k)\oplus R_k$-module
 with $G$-action;
\item $\sum\limits_{k=1}^r\dim ((\ad B_k)\oplus R_k) = d$;
\item $B_i \tilde T_k = R_i \tilde T_k = 0$ for $i > k$.
\end{enumerate}
\end{lemma}
\begin{proof}
By Lemma~\ref{LemmaBcReducible}, $L$ is a completely reducible $\tilde B_0$-module
with $G$-action.
  Therefore, for every $J_k$ we can choose
  a complementary $G$-invariant $\tilde B_0$-submodules
  $\tilde T_k$ in $I_k$. Then $\tilde T_k$ are both $B$- and $\tilde B$-submodules.

Note that $(\ad a_{ij})w=c_{ij}w$ for all $w \in I_k/J_k$
since $I_k/J_k$ is an irreducible $A$-module and $J(A)\,I_k/J_k = 0$.
Hence, by Lemma~\ref{LemmaChooseReduct}, $I_k/J_k$ is a
 faithful $(\ad B_k)\oplus R_k$-module,
  $R_i\, I_k/J_k = 0$ for $i > k$
and the elements $c_{ij}$ are linearly independent.
Moreover, by Property 5 of Lemma~\ref{LemmaChooseReduct},
 $I_k/J_k$ is an irreducible $\left(\sum_{i=1}^r
 (B_i\oplus \tilde R_i)\oplus N
\right)/N$-module with $G$-action.
However
$\left(\sum_{i=1}^r (B_i\oplus \tilde R_i)\oplus N
\right)/N$ acts on $I_k/J_k$
by the same operators as $\tilde B$. Thus
 $\tilde T_k \cong I_k/J_k$  is an irreducible $\tilde B$-module
 with $G$-action. Property 1 is proved. By Lemma~\ref{LemmaBcReducible}, $L$ is a completely reducible $(\ad B_k)\oplus R_k$-module
with $G$-action for any $1 \leqslant k \leqslant r$. Using the isomorphism $\tilde T_k \cong I_k/J_k$, we obtain Properties 2 and 4 from the remarks above.
Property 3 is a consequence of Property 3 of Lemma~\ref{LemmaChooseReduct}.
\end{proof}

\begin{lemma}\label{ChooseSubmodule} For all $1 \leqslant k \leqslant r$
we have
 $$\tilde T_k = T_{k1} \oplus T_{k2} \oplus \dots
\oplus T_{km}$$ where $T_{kj}$ are faithful irreducible
$(\ad B_k)\oplus R_k$-submodules with $G$-action, $m \in \mathbb N$,
$1 \leqslant j \leqslant m$.
\end{lemma}
\begin{proof}
By Lemma~\ref{LemmaSiProperties}, Property 2,
$\tilde T_k = T_{k1} \oplus T_{k2} \oplus \dots
\oplus T_{km}$ for some irreducible
$(\ad B_k)\oplus R_k$-submodules with $G$-action.
Suppose $T_{kj}$ is not faithful for some $1 \leqslant j \leqslant m$.
Hence $b T_{kj}=0$ for some $b \in (\ad B_k)\oplus R_k$,
$b \ne 0$.
Note that $\tilde B = ((\ad B_k)\oplus R_k) \oplus \tilde B_k$
where $$\tilde B_k := \bigoplus_{i\ne k} (\ad B_i )\oplus
\bigoplus_{i\ne k} R_i$$ and $[(\ad B_k)\oplus R_k, \tilde B_k]=0$.
Denote by $\widehat B_k$ the associative subalgebra
of $\End_F(\tilde T_k)$ with $1$
generated by operators from $\tilde B_k$.
Then $$[(\ad B_k)\oplus R_k, \widehat B_k]=0$$ and $\sum_{a \in \widehat B_k}
a T_{kj} \supseteq T_{kj}$ is a $G$-invariant $\tilde B$-submodule
of $\tilde T_k$ since $$\left(\sum_{a \in \widehat B_k}
a T_{kj}\right)^g = \sum_{a \in \widehat B_k}
a^g T_{kj}^g = \sum_{a \in \widehat B_k}
a^g T_{kj} = \sum_{a' \in \widehat B_k}
a' T_{kj}$$ for all $g\in G$. Thus
$ \tilde T_k = \sum_{a \in \widehat B_k}
a T_{kj}$ and $$b \tilde T_k
= \sum_{a \in \widehat B_k}
ba T_{kj} = \sum_{a \in \widehat B_k}
a (bT_{kj})=0.$$ We get a contradiction with faithfulness
of $\tilde T_{k}$.
\end{proof}

 By Condition~2 of the definition of $d$,
 there exist numbers $q_1, \ldots, q_{r} \in \mathbb Z_+$
such that
$$[[\tilde T_1, \underbrace{L, \ldots, L}_{q_1}], [\tilde T_2, \underbrace{L, \ldots, L}_{q_2}] \ldots, [\tilde T_r,
 \underbrace{L, \ldots, L}_{q_r}]] \ne 0$$
 Choose $n_i \in \mathbb Z_+$ with the maximal $\sum\limits_{i=1}^r n_i$ such that
$$[[\left(\prod_{k=1}^{n_1} j_{1k}\right)\tilde T_1, \underbrace{L, \ldots, L}_{q_1}],
 [\left(\prod_{k=1}^{n_2} j_{2k}\right) \tilde T_2, \underbrace{L, \ldots, L}_{q_2}] \ldots, [\left(\prod_{k=1}^{n_r} j_{rk}\right) \tilde T_r,
 \underbrace{L, \ldots, L}_{q_r}]] \ne 0
$$ for some $j_{ik}\in J(A)$.
Let $j_i := \prod_{k=1}^{n_i} j_{ik}$.
Then $j_i  \in J(A) \cup \{1\}$ and
$$[[j_1 \tilde T_1, \underbrace{L, \ldots, L}_{q_1}], [j_2 \tilde T_2, \underbrace{L, \ldots, L}_{q_2}], \ldots, [j_r \tilde T_r,
 \underbrace{L, \ldots, L}_{q_r}]] \ne 0,
 $$ but
\begin{equation}\label{EquationJZero}
[[j_1 \tilde T_1, \underbrace{L, \ldots, L}_{q_1}],
\ldots, [j_k (j \tilde T_k), \underbrace{L, \ldots, L}_{q_k}], \ldots, [j_r \tilde T_r,
 \underbrace{L, \ldots, L}_{q_r}]] = 0
\end{equation}
for all $j \in J(A)$ and $1 \leqslant k \leqslant r$.

   In virtue of Lemma~\ref{ChooseSubmodule}, for every $k$ we can choose a
   faithful irreducible
 $(\ad B_k)\oplus R_k$-submodule with $G$-action
$T_k \subseteq \tilde T_k$
such that\begin{equation}\label{EquationqNonZero}
[[j_1 T_1, \underbrace{L, \ldots, L}_{q_1}], [j_2 T_2, \underbrace{L, \ldots, L}_{q_2}] \ldots, [j_r T_r, \underbrace{L, \ldots, L}_{q_r}]] \ne 0.
\end{equation}

\begin{lemma}\label{LemmaChange}
Let $\psi \colon \bigoplus_{i=1}^r(B_i \oplus \tilde R_i) \to
\bigoplus_{i=1}^r((\ad B_i)\oplus R_i) $
be the linear isomorphism defined by formulas $\psi(b)= \ad b$ for
all $b \in B_i$ and $\psi(a_{i\ell})=c_{i\ell}$, $1 \leqslant \ell
\leqslant k_\ell$.
Let $f_i$ be multilinear associative $G$-polynomials,
$h^{(i)}_1, \ldots, h^{(i)}_{n_i}
\in \bigoplus_{i=1}^r B_i \oplus \tilde R_i$, $\bar t_i \in \tilde T_i$, $\bar u_{ik}\in L$,
 be some elements.
Then
$$[[j_1 f_1(\ad h^{(1)}_1, \ldots, \ad h^{(1)}_{n_1})
 \bar t_1, \bar u_{11}, \ldots, \bar u_{1q_1}],  \ldots, [j_r
 f_r(\ad h^{(r)}_1, \ldots, \ad h^{(r)}_{n_r}) \bar t_r,
 \bar u_{r1}, \ldots, \bar u_{rq_r}]]=$$ $$
 [[j_1 f_1(\psi (h^{(1)}_1), \ldots, \psi (h^{(1)}_{n_1}))
 \bar t_1, \bar u_{11}, \ldots, \bar u_{1q_1}],  \ldots,$$ $$ [j_r
 f_r(\psi(h^{(r)}_1), \ldots, \psi (h^{(r)}_{n_r})) \bar t_r,
 \bar u_{r1}, \ldots, \bar u_{rq_r}]].$$ In other words, we can replace $\ad a_{i\ell}$ with $c_{i\ell}$ and the result does not change.
\end{lemma}
\begin{proof}
We rewrite $\ad a_{i\ell}=c_{i\ell}+d_{i\ell}=\psi(a_i)+d_{i\ell}$ and use the multilinearity
of $f_i$. By~(\ref{EquationJZero}), terms with $d_{i\ell}$ vanish.
\end{proof}

Denote by $A_3 \subseteq \End_F(L)$ the linear span of products
of operators from $\ad L$ and $G$ such that each product
contains at least one element from $\ad L$.

\begin{lemma}\label{LemmaA3}
$J(A) \subseteq A_3$.
\end{lemma}
\begin{proof}
Note that $A_3$ is a $G$-invariant two-sided ideal of $A$
and $A_3 + \tilde A_3 = A$
where $\tilde A_3 \subseteq \End_F(L)$ is the associative subalgebra
generated by operators from $G$.
Thus $A/A_3 \cong \tilde A_3/(\tilde A_3 \cap A_3)$
is a semisimple algebra since $\tilde A_3$ is a homomorphic image
of the semisimple group algebra $FG$.
Thus $J(A) \subseteq A_3$.
\end{proof}

\begin{lemma}\label{LemmaAlt} If $d \ne 0$, then there exist a number $n_0 \in \mathbb N$ such that for every $n\geqslant n_0$
there exist disjoint subsets $X_1$, \ldots, $X_{2k} \subseteq \lbrace x_1, \ldots, x_n
\rbrace$, $k := \left[\frac{n-n_0}{2d}\right]$,
$|X_1| = \ldots = |X_{2k}|=d$ and a polynomial $f \in V^G_n \backslash
\Id^G(L)$ alternating in the variables of each set $X_j$.
\end{lemma}

\begin{proof}
Denote by $\varphi_i \colon (\ad B_i)\oplus R_i \to
\mathfrak{gl}(T_i)$ the representation
 corresponding to the action of $(\ad B_i)\oplus R_i$
on $T_i$.
In virtue of Theorem~\ref{TheoremAlternateFinal},
there exist constants $m_i \in \mathbb Z_+$
such that for any $k$ there exist
 multilinear polynomials $f_i \in Q^G_{d_i, 2k, 2k d_i+m_i}
  \backslash \Id^G(\varphi_i)$,
$d_i := \dim ((\ad B_i)\oplus R_i)$,
alternating in the variables from disjoint sets
$X^{(i)}_{\ell}$, $1 \leqslant \ell \leqslant 2k$, $|X^{(i)}_{\ell}|=d_i$.

In virtue of~(\ref{EquationqNonZero}),
$$[[j_1 \bar t_1, \bar u_{11}, \ldots, \bar u_{1,q_1}], [j_2 \bar t_2, \bar u_{21}, \ldots, \bar u_{2,q_2}],
 \ldots, [j_r \bar t_r, \bar u_{r1}, \ldots, \bar u_{r,q_r}]] \ne 0,
 $$
 for some $\bar u_{i\ell} \in L$ and $\bar t_i \in T_i$. All $j_i \in J(A)\cup \{1\}$ are polynomials in elements from $G$ and $\ad L$. Denote by $\tilde m$ the maximal degree of them.

 Recall that each $T_i$ is a faithful irreducible $(\ad B_i)\oplus R_i$-module
with $G$-action.
Therefore by the density theorem,
 $\End_F(T_i)$ is generated by operators from $G$
and $(\ad B_i)\oplus R_i$.
Note that $\End_F(T_i) \cong M_{\dim T_i}(F)$.
Thus every matrix unit $e^{(i)}_{j\ell} \in M_{\dim T_i}(F)$ can be
represented as a polynomial in operators from $G$
and $(\ad B_i)\oplus R_i$. Choose such polynomials
for all $i$ and all matrix units. Denote by $m_0$ the maximal degree of those
polynomials.

Let $n_0 := r(2m_0+\tilde m+1)+ \sum_{i=1}^r (m_i+q_i)$.
Now we choose $f_i$ for $k = \left[\frac{n-n_0}{2d}\right]$.
Since $f_i \notin \Id^G(\varphi_i)$,
there exist $\bar x_{i1}, \ldots, \bar x_{i, 2k d_i+m_i} \in (\ad B_i)\oplus R_i$
such that $f_i(\bar x_{i1}, \ldots, \bar x_{i, 2k d_i+m_i})\ne 0$.
Hence $$e^{(i)}_{\ell_i \ell_i} f_i(\bar x_{i1}, \ldots, \bar x_{i, 2k d_i+m_i})
e^{(i)}_{s_i s_i} \ne 0$$ for some matrix units $e^{(i)}_{\ell_i \ell_i},
e^{(i)}_{s_i s_i} \in \End_F(T_i)$, $1 \leqslant \ell_i, s_i \leqslant \dim {T_i}$.
Thus $$\sum_{\ell=1}^{\dim_{T_i}}
e^{(i)}_{\ell \ell_i} f_i(\bar x_{i1}, \ldots, \bar x_{i, 2k d_i+m_i})
 e^{(i)}_{s_i \ell}$$ is a nonzero scalar operator in $\End_F(T_i)$.

Hence
$$ [[j_1\left(\sum_{\ell=1}^{\dim {T_1}}
e^{(1)}_{\ell \ell_1} f_1(\bar x_{11}, \ldots, \bar x_{1,2k d_1+m_1})
 e^{(1)}_{s_1 \ell}\right)\bar t_1, \bar u_{11}, \ldots, \bar u_{1q_1}],
 \ldots, $$
 $$
 [j_r\left(\sum_{\ell=1}^{\dim {T_r}}
e^{(r)}_{\ell \ell_r} f_r(\bar x_{r1}, \ldots, \bar x_{r, 2k d_r+m_r})
 e^{(r)}_{s_r \ell}\right)\bar t_r, \bar u_{r1}, \ldots, \bar u_{rq_r}]]\ne 0.$$
Denote $X_\ell := \bigcup_{i=1}^{r} X^{(i)}_{\ell}$.
Let $\Alt_\ell$ be the operator of alternation
in the variables from $X_\ell$.
Consider
$$\tilde f(x_{11}, \ldots, x_{1, 2k d_1+m_1},
\ \ldots,\ x_{r1}, \ldots, x_{r, 2k d_r+m_r}) :=
$$ $$
 \Alt_1 \Alt_2 \ldots \Alt_{2k} [[j_1\left(\sum_{\ell=1}^{\dim {T_1}}
e^{(1)}_{\ell \ell_1} f_1(x_{11}, \ldots, x_{1, 2k d_1+m_1})
 e^{(1)}_{s_1 \ell}\right)\bar t_1, \bar u_{11}, \ldots, \bar u_{1q_1}],
 \ldots, $$
 $$
 [j_r\left(\sum_{\ell=1}^{\dim {T_r}}
e^{(r)}_{\ell \ell_r} f_r(x_{r1}, \ldots, x_{r,2k d_r+m_r})
 e^{(r)}_{s_r \ell}\right)\bar t_r, \bar u_{r1}, \ldots, \bar u_{rq_r}]].$$
Then
$$\tilde f(\bar x_{11}, \ldots, \bar x_{1, 2k d_1+m_1},
\ \ldots,\ \bar x_{r1}, \ldots, \bar x_{r, 2k d_r+m_r})
= $$ $$(d_1!)^{2k} \ldots (d_r!)^{2k} [[j_1\left(\sum_{\ell=1}^{\dim {T_1}}
e^{(1)}_{\ell \ell_1} f_1(\bar x_{11}, \ldots, \bar x_{1,2k d_1+m_1})
 e^{(1)}_{s_1 \ell}\right)\bar t_1, \bar u_{11}, \ldots, \bar u_{1q_1}],
 \ldots, $$
 $$
 [j_r\left(\sum_{\ell=1}^{\dim {T_r}}
e^{(r)}_{\ell \ell_r} f_r(\bar x_{r1}, \ldots, \bar x_{r, 2k d_r+m_r})
 e^{(r)}_{s_r \ell}\right)\bar t_r, \bar u_{r1},
  \ldots, \bar u_{rq_r}]]\ne 0.$$
since $f_i$ are alternating in each $X^{(i)}_{\ell}$
and, by Lemma~\ref{LemmaSiProperties}, $((\ad B_i)\oplus R_i)\tilde T_\ell = 0$
for $i > \ell$. Now we rewrite
$e^{(i)}_{\ell j}$ as polynomials in elements of $(\ad B_i)\oplus R_i$
and $G$.
Using linearity of $\tilde f$ in $e^{(i)}_{\ell j}$,
we can replace $e^{(i)}_{\ell j}$ with the products
of elements from $(\ad B_i)\oplus R_i$
and $G$, and the expression will not vanish
for some choice of the products. Using~(\ref{EqGMoveRight}),
we can move all $\rho(g)$ to the right.
 By Lemma~\ref{LemmaChange},
we can replace all elements from $(\ad B_i)\oplus R_i$
with elements from $B_i\oplus \tilde R_i$
and the expression will be still nonzero.
Denote by $\psi \colon \bigoplus_{i=1}^r (B_i \oplus \tilde R_i) \to
\bigoplus_{i=1}^r ((\ad B_i)\oplus R_i) $ the corresponding linear isomorphism.
Now we rewrite $j_i$ as polynomials in elements $\ad L$ and $G$.
Since $\tilde f$ is linear in $j_i$,
we can replace $j_i$ with one of the monomials,
i.e. with the product of elements from $\ad L$ and $G$.
Using~(\ref{EqGMoveRight}),
we again move all $\rho(g)$ to the right. Then
we replace the elements from $\ad L$ with new variables,
and
$$\hat f :=
 \Alt_1 \Alt_2 \ldots \Alt_{2k} \biggl[\Bigl[\Bigl[y_{11}, [y_{12}, \ldots
 [y_{1 \alpha_1}, \Bigl[z_{11}, [z_{12},
 \ldots, [z_{1 \beta_1},
$$ $$
  (f_1(\ad x_{11}, \ldots, \ad x_{1, 2k d_1+m_1}))^{g_1}
 [w_{11}, [w_{12}, \ldots, [w_{1 \gamma_1},
  t_1^{h_1}]\ldots \Bigr],
  u_{11}, \ldots, u_{1q_1}\Bigr],
 \ldots, $$
 $$\Bigl[\Bigl[y_{r1}, [y_{r2}, \ldots,
 [y_{r \alpha_r}, \Bigr[z_{r1},
  [z_{r2},
 \ldots, [z_{r \beta_r},
 $$ $$
 (f_r(\ad x_{r1}, \ldots, \ad x_{r, 2k d_r+m_r}))^{g_r}
 [w_{r1}, [w_{r2}, \ldots, [w_{r \gamma_r}, t_r^{h_r}]\ldots \Bigr],
  u_{r1}, \ldots, u_{rq_r}\Bigr]\biggr]$$
  for some  $0 \leqslant \alpha_i \leqslant \tilde m$,
  \quad
  $0 \leqslant \beta_i, \gamma_i \leqslant m_0$,
  \quad $g_{i}, h_i \in G$,\quad
  $\bar y_{i\ell}, \bar z_{i\ell},
  \bar w_{i\ell} \in L$
 does not vanish under the substitution
 $t_i=\bar t_i$, $u_{i\ell}=\bar u_{i\ell}$,
 $x_{i\ell}=\psi^{-1}(\bar x_{i\ell})$, $y_{i\ell}=\bar y_{i\ell}$,
 $z_{i\ell}=\bar z_{i\ell}$, $w_{i\ell}=\bar w_{i\ell}$.

Note that $\hat f \in V_{\tilde n}^G$,
  $\tilde n: = 2kd +r+ \sum_{i=1}^r (m_i + q_i + \alpha_i+\beta_i+\gamma_i)
  \leqslant n$. If $n=\tilde n$, then we take $f:=\hat f$.
  Suppose $n > \tilde n$.
  Let $b \in (\ad B_1) \oplus R_1$, $b\ne 0$.
Then $e^{(1)}_{jj} b e^{(1)}_{\ell\ell} \ne 0$
for some $1 \leqslant j,\ell \leqslant \dim T_1$
and $\left(\sum_{s=1}^{\dim T_1}(e^{(1)}_{sj} b
 e^{(1)}_{\ell s})\right)^{n-\tilde n}\bar t_1 =
\mu \bar t_1$, $\mu \in F\backslash \{0\}$.
Hence $\hat f$ does not vanish under the substitution
 $t_1 = \left(\sum_{s=1}^{\dim T_1}(e^{(1)}_{sj} b
 e^{(1)}_{\ell s})\right)^{n-\tilde n}\bar t_1$; $t_i=\bar t_i$ for $2 \leqslant i \leqslant r$; $u_{i\ell}=\bar u_{i\ell}$,
 $x_{i\ell}=\psi^{-1}(\bar x_{i\ell})$, $y_{i\ell}=\bar y_{i\ell}$,
 $z_{i\ell}=\bar z_{i\ell}$, $w_{i\ell}=\bar w_{i\ell}$.

 By Lemma~\ref{LemmaA3},
  $$b \in J(A) \oplus \ad (B_1 \oplus \tilde R_1)\subseteq A_3$$
  and using~(\ref{EqGMoveRight})
 we can rewrite $\left(\sum_{s=1}^{\dim T_1}(e^{(1)}_{sj} b
 e^{(1)}_{\ell s})\right)^{n-\tilde n}\bar t_1$
 as a sum of elements $[\bar v_1, [\bar v_2, [\ldots,
  [\bar v_q,  \bar t_1^g]\ldots]$,
 $q \geqslant n-\tilde n$, $\bar v_i  \in L$, $g\in G$.
 Hence $\hat f$ does not vanish under a substitution
 $t_1 = [\bar v_1, [\bar v_2, [\ldots, [\bar v_q,  \bar t_1^g]\ldots]$
 for some $q \geqslant n-\tilde n$, $\bar v_i \in L$, $g\in G$;
  $t_i=\bar t_i$ for $2 \leqslant i \leqslant r$; $u_{i\ell}=\bar u_{i\ell}$,
 $x_{i\ell}=\psi^{-1}(\bar x_{i\ell})$, $y_{i\ell}=\bar y_{i\ell}$,
 $z_{i\ell}=\bar z_{i\ell}$, $w_{i\ell}=\bar w_{i\ell}$.
Therefore, $$f :=
 \Alt_1 \Alt_2 \ldots \Alt_{2k} \biggl[\Bigl[\Bigl[y_{11}, [y_{12}, \ldots
 [y_{1 \alpha_1}, \Bigl[z_{11}, [z_{12},
 \ldots, [z_{1 \beta_1},
$$ $$
  (f_1(\ad x_{11}, \ldots, \ad x_{1, 2k d_1+m_1}))^{g_1}
 [w_{11}, [w_{12}, \ldots, [w_{1 \gamma_1},
  $$
 $$\bigl[v_1^{h_1}, [v_2^{h_1}, [\ldots, [v_{n-\tilde n}^{h_1}, t_1^{h_1}]\ldots\bigr]\ldots \Bigr],
  u_{11}, \ldots, u_{1q_1}\Bigr],$$
 $$
  \Bigl[\Bigl[y_{21}, [y_{22}, \ldots
 [y_{2 \alpha_2}, \Bigl[z_{21}, [z_{22},
 \ldots, [z_{2 \beta_2},
$$ $$
  (f_2(\ad x_{21}, \ldots, \ad x_{2, 2k d_2+m_2}))^{g_2}
 [w_{21}, [w_{22}, \ldots, [w_{2 \gamma_2},
  t_2^{h_2}]\ldots \Bigr],
  u_{21}, \ldots, u_{2q_2}\Bigr],  $$
$$
 \ldots, \Bigl[\Bigl[y_{r1}, [y_{r2},\ldots,
 [y_{r \alpha_r}, \Bigr[z_{r1},
  [z_{r2},
 \ldots, [z_{r \beta_r},
 $$ $$
 (f_r(\ad x_{r1}, \ldots, \ad x_{r, 2k d_r+m_r}))^{g_r}
 [w_{r1}, [w_{r2}, \ldots, [w_{r \gamma_r}, t_r^{h_r}]\ldots \Bigr],
  u_{r1}, \ldots, u_{rq_r}\Bigr]\biggr]$$
  does not vanish under the substitution
  $v_\ell = \bar v_\ell$, $1 \leqslant \ell \leqslant n-\tilde n$,
  $t_1 = [\bar v_{n-\tilde n +1}, [\bar v_{n-\tilde n +2}, [\ldots, [\bar v_q,  \bar t_1^g]\ldots]$;
  $t_i=\bar t_i$ for $2 \leqslant i \leqslant r$; $u_{i\ell}=\bar u_{i\ell}$,
 $x_{i\ell}=\psi^{-1}(\bar x_{i\ell})$, $y_{i\ell}=\bar y_{i\ell}$,
 $z_{i\ell}=\bar z_{i\ell}$, $w_{i\ell}=\bar w_{i\ell}$.
 Note that $f \in V_n^G$ and satisfies all the conditions of the lemma.
\end{proof}

\begin{lemma}\label{LemmaCochar} Let
 $k, n_0$ be the numbers from
Lemma~\ref{LemmaAlt}.   Then for every $n \geqslant n_0$ there exists
a partition $\lambda = (\lambda_1, \ldots, \lambda_s) \vdash n$,
$\lambda_i \geqslant 2k-C$ for every $1 \leqslant i \leqslant d$,
with $m(L, G, \lambda) \ne 0$.
Here $C := p((\dim L)p + 3)((\dim L)-d)$ where $p \in \mathbb N$ is such number that $N^p=0$.
\end{lemma}
\begin{proof}
Consider the polynomial $f$ from Lemma~\ref{LemmaAlt}.
It is sufficient to prove that $e^*_{T_\lambda} f \notin \Id^G(L)$
for some tableau $T_\lambda$ of the desired shape $\lambda$.
It is known that $FS_n = \bigoplus_{\lambda,T_\lambda} FS_n e^{*}_{T_\lambda}$ where the summation
runs over the set of all standard tableax $T_\lambda$,
$\lambda \vdash n$. Thus $FS_n f = \sum_{\lambda,T_\lambda} FS_n e^{*}_{T_\lambda}f
\not\subseteq \Id^G(L)$ and $e^{*}_{T_\lambda} f \notin \Id^G(L)$ for some $\lambda \vdash n$.
We claim that $\lambda$ is of the desired shape.
It is sufficient to prove that
$\lambda_d \geqslant 2k-C$, since
$\lambda_i \geqslant \lambda_d$ for every $1 \leqslant i \leqslant d$.
Each row of $T_\lambda$ includes numbers
of no more than one variable from each $X_i$,
since $e^{*}_{T_\lambda} = b_{T_\lambda} a_{T_\lambda}$
and $a_{T_\lambda}$ is symmetrizing the variables of each row.
Thus $\sum_{i=1}^{d-1} \lambda_i \leqslant 2k(d-1) + (n-2kd) = n-2k$.
In virtue of Lemma~\ref{LemmaUpper},
$\sum_{i=1}^d \lambda_i \geqslant n-C$. Therefore
$\lambda_d \geqslant 2k-C$.
\end{proof}

\begin{proof}[Proof of Theorem~\ref{TheoremMain}]
The Young diagram~$D_\lambda$ from Lemma~\ref{LemmaCochar} contains
the rectangular subdiagram~$D_\mu$, $\mu=(\underbrace{2k-C, \ldots, 2k-C}_d)$.
The branching rule for $S_n$ implies that if we consider the restriction of
$S_n$-action on $M(\lambda)$ to $S_{n-1}$, then
$M(\lambda)$ becomes the direct sum of all non-isomorphic
$FS_{n-1}$-modules $M(\nu)$, $\nu \vdash (n-1)$, where each $D_\nu$ is obtained
from $D_\lambda$ by deleting one box. In particular,
$\dim M(\nu) \leqslant \dim M(\lambda)$.
Applying the rule $(n-d(2k-C))$ times, we obtain $\dim M(\mu) \leqslant \dim M(\lambda)$.
By the hook formula, $$\dim M(\mu) = \frac{(d(2k-C))!}{\prod_{i,j} h_{ij}}$$
where $h_{ij}$ is the length of the hook with edge in $(i, j)$.
By Stirling formula,
$$c_n^G(L)\geqslant \dim M(\lambda) \geqslant \dim M(\mu) \geqslant \frac{(d(2k-C))!}{((2k-C+d)!)^d}
\sim $$ $$\frac{
\sqrt{2\pi d(2k-C)} \left(\frac{d(2k-C)}{e}\right)^{d(2k-C)}
}
{
\left(\sqrt{2\pi (2k-C+d)}
\left(\frac{2k-C+d}{e}\right)^{2k-C+d}\right)^d
} \sim C_9 k^{r_9} d^{2kd}$$
for some constants $C_9 > 0$, $r_9 \in \mathbb Q$,
as $k \to \infty$.
Since $k = \left[\frac{n-n_0}{2d}\right]$,
this gives the lower bound.
The upper bound has been proved in Theorem~\ref{TheoremUpper}.
\end{proof}

\section{Acknowledgements}

I am grateful to Yuri Bahturin and Mikhail Kotchetov for helpful
discussions.

\end{document}